\documentclass[11pt]{amsart}
\usepackage{amsmath}
\usepackage{amssymb}
\usepackage{eucal}
\usepackage{graphicx}
\usepackage{color}
\usepackage{setspace,cite}
\usepackage[colorlinks]{hyperref}
\usepackage[active,new,noold,marker]{xrcs}
\usepackage[active]{xrcs}
\usepackage{hyperref}
\usepackage[active]{srcltx}
\newtheorem{Definition}{Definition}[subsection]
\newtheorem{Theorem}[Definition]{Theorem}
\newtheorem{Lemma}[Definition]{Lemma}

\newtheorem{Proposition}[Definition]{Proposition}

\newtheorem{Remark}[Definition]{Remark}

\newcommand{\bt}{\begin{Theorem}}
\newcommand{\et}{\end{Theorem}}
\newcommand{\ba}{\begin{eqnarray}}
\newcommand{\ea}{\end{eqnarray}}
\newcommand{\bd}{\begin{Definition}}
\newcommand{\ed}{\end{Definition}}
\newcommand{\bp}{\begin{Proposition}}
\newcommand{\ep}{\end{Proposition}}
\newcommand{\bl}{\begin{Lemma}}
\newcommand{\el}{\end{Lemma}}
\newcommand{\br}{\begin{Remark}}
\newcommand{\er}{\end{Remark}}
\newcommand{\bpf}{\begin{proof}}
\newcommand{\epf}{\end{proof}}
\newcommand{\be}{\begin{equation}}
\newcommand{\ee}{\end{equation}}
\newcommand{\mc}{\mathcal}

\newcommand{\f}{\frac}

\newcommand{\what}{\widehat}

\newcommand{\mf}{\mathfrak}
\newcommand{\R}{\mathbb R}%
\newcommand{\C}{\mathbb C}%
\newcommand{\Z}{\mathbb Z}%
\newcommand{\lt}{\left}%
\newcommand{\rt}{\right}%
\newcommand{\e}{\varepsilon}%
\numberwithin{equation}{section}
\numberwithin{Definition}{section}

\allowdisplaybreaks[1]

\begin{document}
 \title [ Schwartz space isomorphism theorem ]{ On the Schwartz space isomorphism theorem for the Riemannian symmetric spaces}
\author [ Jana ]{  Joydip Jana }
\address [ Joydip Jana]{ Syamaprasad College \\Department of Mathematics\\92, S. P. Mukherjee Road\\Kolkata-700 026, India\\ E-mail : joydipjana@gmail.com}
\keywords{ Schwartz spaces, $\delta$-spherical transform, Helgason Fourier transform}
\thanks{\emph{ Mathematical Subject Classification}: 43A80, 43A85, 43A90}
 \begin{abstract}
 We deduce a proof of the isomorphism theorem for certain closed subspace $\mc S^p_\Gamma(X)$ of the $L^p$-Schwartz class functions $(0< p \leq 2) $ on a Riemannian symmetric space $X$ where $\Gamma$ is a finite subset of $\what{K}_M$. The Fourier transform considered is the Helgason Fourier transform.
Our proof  relies only on the Paley-Wiener theorem for the corresponding class of functions and hence it does not use the complicated higher asymptotics of the elementary spherical functions.
\end{abstract}
\maketitle
 \section*{\textbf{ Introduction}}\label{sec:Introdu}
\setcounter{equation}{0}
Let $X$ be a Riemannian symmetric space realized as $G/K$, where $G$ is a connected, noncompact  semisimple Lie group with finite center. Let us fix  $K$ is a maximal compact subgroup of $G$. The $L^p$-Schwartz space isomorphism theorem  for bi-$K$-invariant functions on the group $G$ under the spherical transform was first proved by Harish-Chandra \cite{Harish58(I),Harish58(II),Harish66} (for $p=2$), Trombi and Varadarajan \cite{Varadarajan-T71} (for $0 < p < 2$). Recently Anker \cite{Anker91} gave a remarkable short and elegant proof of the above theorem for $0 < p \leq 2$. Anker's work does not involve the asymptotic expansion of the elementary spherical functions which has a crucial role in the earlier works. 
The aim of this paper is to extend Anker's technique for the $L^p$-Schwartz class functions (for $0< p \leq 2$) on $X$ and to establish an isomorphism theorem under the Helgason fourier transform (HFT).\\
The main result of this paper is developed in two theorems Theorem \ref{Theorem:main-1} and Theorem \ref{Theo:main theorem 2}. In Theorem \ref{Theorem:main-1} the basic $L^p$-Schwartz space $\mc S^p_\delta(X)$ is the space of operator valued left-$\delta$-type $(\delta \in \what{K}_M )$, smooth functions on $X$. HFT when restricted to $\mc S^p_\delta(X)$ is identified with the `$\delta$-spherical transform'. In Theorem \ref{Theorem:main-1} we establish an isomorphism between the Schwartz spaces $\mc S^p_\delta(X)$ and $\mc S_\delta(\mf a^*_\e)$ under the $\delta$-spherical transform. The image $\mc S_\delta(\mf a^*_\e)$ is a space of matrix valued functions with certain decay defined on a closed complex tube $\mf a^*_\e = \mf a^* + i C^{\e \rho}$. Explicit definition of the space $\mc S_\delta(\mf a^*_\e) $ and the domain $\mf a^*_\e$ will be given in the next section. Restriction of Theorem \ref{Theorem:main-1} to the rank-one case is a part of the result of Eguchi and Kowata \cite{Eguchi76}. Theorem \ref{Theorem:main-1} also relaxes the rank restriction of the similar result obtained in \cite{Jana}.\\
Theorem \ref{Theo:main theorem 2} further extends the isomorphism obtained in Theorem \ref{Theorem:main-1} to the space $\mc S^p_\Gamma(X)$ of scalar valued $K$-finite Schwartz class functions on $X$ for which the left $K$-types lie in a fixed finite subset $\Gamma \subset \what{K}_M$. The transform considered in Theorem \ref{Theo:main theorem 2} is the HFT.\\
We shall closely follow the notations of \cite{Helgason-gga, Helgason-gas}. Some basic definitions and results used in this paper are given in the following section.

 \section{\textbf{Notation and Preliminaries}}\label{sec:Prelimi}
Let $G$ be a connected, noncompact semisimple Lie group with finite center and $K$ be a maximal compact subgroup of $G$. Let $\theta$ be the Cartan involution corresponding to $K$.  Let $X$ be a Riemannian symmetric space realized as $G/K$.
Let $\mf g$ be the Lie algebra of $G$ and $\mf g_\C$ be its complexification. The Hermitian norm of both $\mf g$ and $\mf g_\C$ will be denoted by the notation $\|\cdot \|$.
We denote $\mathfrak k$ for the Lie algebra of the maximal compact subgroup $K$ of $G$. Let $\mathfrak g= \mathfrak k \oplus \mathfrak s$ be the Cartan decomposition of the Lie algebra.
We fix a maximal abelian subspace $\mathfrak a$ in $\mathfrak s$. Let $A$ be analytic subgroup of $G$ with the Lie algebra $\mathfrak a$. Let $\mathfrak a^*$  and $\mathfrak a^*_\C$ respectively  be the real dual of $\mf a$ and its  complexification. The Killing form induces a scalar product on $\mathfrak a$ and hence on $\mathfrak a^*$. We shall denote  $\langle \cdot, \cdot \rangle_1$ for the $\C$-bilinear extension of that scalar product to $\mathfrak a^*_\C$.  A semisimple Lie group $G$ is said to be of `real rank-$n$' if $dim \mathfrak a= n$ and the corresponding symmetric spaces $X$ realized as $X=G/K$ are called `rank-$n$' symmetric spaces.\\
Let $\Sigma$ be the root system associated with the pair $(\mf g, \mf a)$. For each $\alpha \in \Sigma$ we write $\mf g_\alpha$ for the corresponding root space. Let $M'$ and $M$ respectively be the normalizer and the centralizer of $A$ in $K$. The quotient $W= M'/M$ be the Weyl group associated with the root system $\Sigma$.

  Let us choose and fix a system of positive roots which we denote by $\Sigma^+ $. Let $\mathfrak a^+$ be the corresponding positive Weyl chamber and $\overline{\mathfrak a^+}$ be its closer. We denote $\mathfrak a^{*+}$ and $\overline {\mathfrak a^{*+} }$ for the similar chambers in $\mathfrak a^*$. We put $A^+ = \exp \mathfrak a^+ $ and $\overline{A^+} = \exp \overline{\mathfrak a^+}$. The element $\rho \in \mathfrak a^*$ is denoted by
\begin{equation}
 \label{eq:rho}
 \rho(H)= \frac{1}{2} \sum_{\alpha \in \Sigma^+} m_\alpha \alpha(H), ~~\mbox{ where}~ m_\alpha=dim \mathfrak g_\alpha~ \mbox{and}~ H \in \mathfrak a.
\end{equation}
Let $\Sigma_0 \subset \Sigma$ be the set of all indivisible roots and $\Sigma_0^+ = \Sigma_0 \cap \Sigma^+$.
$\mathfrak n = \oplus_{\alpha \in \Sigma^+ } \mathfrak g_\alpha $ is a nilpotent subalgebra of $\mathfrak g$. Let $N$ be the nilpotent subgroup of $G$ with the Lie algebra $\mf n$. The Iwasawa decomposition of the group $G$ is given as $G = K A N$. The map $(k,a,n)\mapsto kan$ is a diffeomorphism from $K \times A \times N$ onto $G$. Let $\mathcal H: G \rightarrow \mf a$ and $\mathcal A:G \rightarrow \mf a$ are the $\mathfrak a$ projections of $g \in G$ in Iwasawa $KAN$ and $NAK$ decompositions respectively. These two projections are related by $\mc A(g)= -\mc  H(g^{-1})$ for each $g \in G$. Any element $g \in G$ can therefore be written as $g = k \exp {\mc H (g) }n= n_1 \exp {\mc A(g)} k_1$. In the Iwasawa $KAN$ decomposition the Haar measure of the group $G$ is given by
\begin{equation}
\label{eq:Haar measure KAN}
\int_G f(g) dg = const. \int_K dk \int_{\mathfrak a^+} e^{2 \rho( \mc H(g))} d\mc H(g) \int_N dn ~f(k \exp{\mc H(g)} n),
\end{equation}
where $const.$ is a normalizing constant. The `Cartan decomposition' of the group is $G = K \overline{A^+} K = K \exp \overline{ \mathfrak a^+} K $. Let $g^+$ denote the $\mathfrak a^+$ component of $g \in G$, and we denote $|g|= \|g^+\|$. We have a basic estimate: for some constant $c > 0$
\begin{equation}
\label{eq:Iwasawa < Cartan}
\|\mc H(g)\| \leq c |g|, \mbox{~for each~ } g \in G.
\end{equation}
The Haar measure for the Cartan decomposition is given by
\begin{equation}
\label{eq:Haar measure KAK}
\int _G f(g) dg = const. \int_K dk \int_{\mathfrak a^+} \Delta(H) dh \int_K dk' f(k \exp H k'), ~~(H \in \overline{\mf a^+} )
\end{equation}
where $\Delta(H)= \prod_{\alpha \in \Sigma^+} \sinh^{m_\alpha}H$ and $const.$ is a positive normalizing constant. We shall be using the following estimate for the density $\Delta(H)$:
\begin{equation}
\label{eq:estimate of KAK Haar density}
0\leq \Delta(H) \leq c e^{2 \rho(H)}, \mbox{~for~} H \in \mathfrak a^+.
\end{equation}
A function $f$ on $G$ is said to be `bi-$K$-invariant' if $f(k_1 g k_2)= f(g)$ for all $k_1, k_2  \in K$ and $g \in G$. We refer a function a function as `right-$K$-invariant' if invariant under the right $K$ action on $G$ that is  $f(gk) = f(g)$ for all $k \in K, g \in G$. Althrough in this paper we shall consider a function on the symmetric space $X=G/K$ as a right-$K$-invariant function on the group $G$. For any function space $\mf F(G)$ on $G$ or $\mf F(G/K)$ on $X$, we shall denote $\mf F(G//K)$ for the corresponding subspace of bi-$K$-invariant functions.\\
We denote $\mc C^\infty(G)$ for the set of all smooth functions on $G$. We fix a basis $\{X_j \}$ for the Lie algebra $\mf g$. Let $\mc U(\mf g)$ be the `universal enveloping algebra' over $\mf g$. Let $D_1\cdots D_m, E_1 \cdots E_n \in \mc U(\mf g)$, then the action of $\mc U(\mf g)$ on a function $f \in  \mc C^\infty(G)$ is defined as follows:
\begin{align*}
\label{action of uni env alg }
&f(D_1 \cdots D_m, x, E_1 \cdots E_n)= \nonumber\\  &\f{d}{dt_1} \hspace{-.07in} \mid_{t_1=0}\hspace{-.05in}\cdots\hspace{-.05in} \f{d}{dt_m}\hspace{-.07in}\mid_{t_m=0}~\f{d}{ds_1}\hspace{-.07in}\mid_{s_1=0}\hspace{-.05in} \cdots \hspace{-.05in}\f{d}{ds_n}\hspace{-.07in}\mid_{s_n=0} \hspace{-.03in}f(\exp{\hspace{-.03in}(t_1D_1)\hspace{-.03in}}\cdots \hspace{-.03in}\exp{\hspace{-.03in}(t_mD_m)} x \exp{\hspace{-.03in}(s_1E_1)\hspace{-.03in}}\hspace{-.02in}\cdots\hspace{-.03in} \exp{\hspace{-.03in}(s_nE_n)\hspace{-.03in}}).
\end{align*}
Let $b_{ij}= \mf B(X_i, X_j)$ and $(b^{ij})$ be the inverse of the matrix $(b_{ij})$. We now define a distinguished element, called the `Casimir element', of $\mc U(\mf g)$ by
$\Omega = \sum_{i,j} b^{ij}X_i X_j$.
The differential operator $\Omega$ lies in the center of $\mc U(\mf g)$. The action of the `Laplace-Beltrami operator' $\mathbf{L}$ on $X$ is defined by the action of $\Omega$:
\begin{equation}
\label{action of Laplace-B same as Casimir}
\mathbf{L} f(xK)= f(x, \Omega),\hspace{.5in}  x \in G.
\end{equation}
Let us briefly describe the method of construction of a family of rank-one symmetric spaces, the rank-one reductions, which are totally geodesic submanifolds of the general rank symmetric space $G/K$ ( see \cite[Ch. IX, \S2]{Helgason-dgls}, \cite[Ch. IV, \S6]{Helgason-gga}, \cite{Knapp03} ).
 Let $\beta$ be an indivisible root of the system $\Sigma$ and $\mf g_{(\beta)}$ be the Lie subalgebra
 of $\mf g$ generated by the root spaces $\mf g_\beta, ~\mf g_{2\beta}, ~\mf g_{-\beta}$ and $\mf g_{-2\beta}$. The subalgebra $\mf g_{(\beta)}$ is stable under the Cartan involution and it is simple. Let $G_{(\beta)}$ be the analytic subgroup of $G$ corresponding to the Lie subalgebra $\mf g_{(\beta)}$. The Iwasawa decomposition of $G_{(\beta)}$ be $G_{(\beta)} = K_{(\beta)} A_{(\beta)} N_{(\beta)}$ where $K_{(\beta)} = K \cap G_{(\beta)}$, $A_{(\beta)} = A \cap G_{(\beta)}$ and $N_{(\beta)} =N \cap G_{(\beta)}$. Also the centralizer of $A_{(\beta)}$ in $K_{(\beta)}$ is $M_{(\beta)} = M \cap G_{(\beta)}$. The abelian subgroup  $A_{(\beta)}$ is one-dimensional and its Lie algebra $\mf a_{(\beta)}$ is generated by the element $H_\beta \in \mf a$ determined by $\lambda(H_\beta) = \langle \lambda, \beta \rangle$. Hence $G_{(\beta)}$ is of real-rank-one and consequently $G_{(\beta)}/ K_{(\beta)} $ is a rank-one symmetric space.\\
 The restricted roots of $G_{(\beta)} $ are $\{\beta, 2\beta, -\beta, -2\beta \}$ or $\{ \beta, -\beta \}$ according as $2\beta \in \Sigma$ or not. Let us further consider $\beta$ to be a positive root of $G_{(\beta)}$ thus the Lie algebra $\mf n_{\beta}$ of $N_{(\beta)}$ is the sum of the root spaces $\mf g_\beta$  and $\mf g_{2\beta}$. We write $\rho_{(\beta)}$ for the $\rho$-function of $\mf g_{(\beta)}$. It can be shown that \cite[B.3, page: 483]{Helgason-gga} $\rho(H_\beta) \geq \rho_{(\beta)}(H_\beta)$ for all $\beta \in \Sigma_0^+$. The equality holds only when $\beta$ is simple and in that case $\rho_{(\beta)}$ is exactly the restriction of $\rho$ to $\mf g_{(\beta)}$. For each $\beta \in \Sigma_0$ the restriction $\lambda_\beta$ of $\lambda \in \mf a^*_\C$ to $\mf a_{(\beta)}$ is given  by the expression $\lambda_\beta = \f{{\langle \lambda, \beta \rangle}_1}{\langle \beta, \beta \rangle} \beta$.

Let $\pi_\lambda$ ($\lambda \in \mathfrak a^*_\C$) be the spherical principal series representations of $G$ realized on the Hilbert space $L^2(K/M)$ and given by the following:
\begin{equation}
\label{eq:spherical principle series }
\left\{\pi_\lambda(g)\zeta \right\}(kM)= e^{-(i \lambda+\rho)\mc H(g^{-1}k)} \zeta(\mc K(g^{-1}k)M),
\end{equation}
where $\lambda \in \mathfrak a^*_\C,~ \zeta \in L^2(K/M)$ and $\mc K(g^{-1}k)$ denotes the $K$ part of $g^{-1}k$ in the Iwasawa $KAN$ decomposition. The spherical functions $\varphi_\lambda(\cdot)$, given by the following formula
\begin{equation}
\label{eq:spherical function}
\varphi_\lambda(g)= \int_K e^{(i \lambda - \rho)\mc H(g^{-1}k)} dk, ~\mbox{where} ~g \in G, ~\lambda \in \mathfrak a^*_\C,
\end{equation}
are the matrix coefficient of the principal series representations. For each $g \in G$ and $\omega \in W$ $\varphi_{\omega \lambda}(g) = \varphi_{\lambda}(g)$. We shall use the following basic estimates for the elementary spherical functions for our purpose.
  \begin{enumerate}
  \item For each $H \in \overline{\mathfrak a^+}$ and $ \lambda \in \overline{{\mathfrak a^*}^+}$, we have
  \begin{equation}
\label{eq:estimate of phi lambda}
0 < |\varphi_{-i \lambda}(\exp H)| \leq e^{\lambda (H)} \varphi_0(\exp H),
\end{equation}
here, $\varphi_0(\cdot)$ is the spherical function corresponding to $\lambda = 0$. For a proof of the above estimate see \cite{Gangolli88}, Proposition 4.6.1.\\
  \item For all $g \in G$, $0 < \varphi_0(g) \leq 1$ \cite[Proposition 4.6.3]{Gangolli88} .
   Also for $\overline{\mf a^+} $ we have the following estimate
\begin{eqnarray}
\label{eq:estimate of phi_0}
e^{\rho(H)} \leq \varphi_0(\exp H) \leq ~ {\ss} (1 + \|H\|)^{c_{\mf a}} ~e^{\rho(H)} ,
\end{eqnarray}
   where, ${\ss}, c_{\mf a} > 0$ are group dependent constants.
This is an work of Harish-Chandra, the above
optimal estimate was obtained by Anker \cite{Anker87}.
\end{enumerate}
Let $\mc D(G)$ be the subspace of $\mc C^\infty(G)$ generated by the compactly supported scalar valued smooth functions on $G$.
 For each $f \in \mathcal D( G//K)$  the `spherical Fourier transform' $\mathcal S f$ is defined by
\begin{equation}
\label{eq:spherical transform}
\mathcal S f(\lambda)= \int_G f(x) \varphi_{-\lambda} (x) dx, ~\mbox{where} ~\lambda \in \mathfrak a^*_\C.
\end{equation}
The inversion of the spherical transform is given by
\begin{equation}
\label{eq:inversion of spherical transform}
f(x)= \frac{1}{|W|} \int_{\mathfrak a^*_\C} \varphi_\lambda(x) ~\mathcal S f(\lambda)~ |\mathbf{c}(\lambda)|^{-2} d \lambda,
\end{equation}
where $|W|$ is the cardinality of the group $W$ and $\mathbf{c}(\lambda)$ is the `Harish-Chandra $\mathbf{c}$-function'. For our purpose we shall use the following estimate  \cite{Anker92}: ~ there exists constants $a, b > 0$ such that
\begin{equation}
\label{eq:c-function estimate}
|\mathbf{c}(\lambda)|^{-2} \leq ~a(\|\lambda\|+1)^b, \mbox{ ~for all~} \lambda \in \mathfrak a^*.
\end{equation}
 For $f \in \mathcal D(X)$, the Helgason Fourier transform (HFT) $\mathcal F f$ is a function on $\mathfrak a^*_\C \times K/M$ and it is defined by ( \cite{Helgason-gas}, Ch. III, $\S$ 1)
\begin{equation}
\label{eq:HFT definition}
\mathcal F f(\lambda , kM)= \int_G f(x) ~e^{(i \lambda - \rho)( H(x, kM))} dx
\end{equation}
where the function $ H: G \times K/M \mapsto \mathfrak a$ is given by $ H(x, kM)= \mc H(x^{-1}k)$. For the sake of simplicity we fix the notational convention $\mc Ff(\lambda, kM) = \mc Ff(\lambda, k)$. We should note that, for a bi-$K$-invariant function  (that is a left-$K$-invariant function) $g$ on $X$,
the HFT reduces to the spherical transform: $\mc Fg(\lambda, k)= \mc Fg(\lambda, e)=\mc Sg(\lambda)$.\\ The inversion formula for HFT \cite[Ch.-III, Theorem 1.3]{Helgason-gas} for $f \in \mathcal D(X)$ is as follows:
\begin{equation}
\label{eq:HFT inversion}
f(x)= \frac{1}{|W|} \int_{\mathfrak a^*} \int_K \mathcal F f(\lambda , k) ~e^{-(i \lambda +\rho)(\mc H(x^{-1}k))} |\mathbf{c}(\lambda)|^{-2} d \lambda~dk.
\end{equation}
Let $\delta$ be a unitary irreducible representation of K realized on a finite dimensional vector space $V_\delta$ with an inner product $\langle \cdot, \cdot \rangle$.  Let us denote $dim V_\delta = d_\delta$. We denote by $\what{K} $  the set of equivalence classes of unitary irreducible representations of $K$ and by customary abuse of notation regard each element of $\what{K}$ as a representation from its equivalence class. For each $\delta \in \what{K}$, let $\chi_\delta$ stand for
the character of the representation $\delta $ and  $ V_\delta^M= \{ v \in V_\delta ~| ~\delta(m)v = v \mbox{~for all~} m \in M\} $ is
the subspace of $V_\delta$ fixed under $\delta|_M$. Let $\widehat{K}_M$ stands for the subset of $\what{K}$ consisting of $\delta$ for which $V_\delta^M \neq \{0\}$ and we will mostly be interested in representations $\delta \in
\widehat{K}_M$. We set an orthogonal basis $\{v_j \}_{1 \leq j \leq d_\delta}$ of $V_\delta$ and we assume that  $\{v_1, \cdots v_{\ell_\delta}\} $ generates $V_\delta^M$ where $dim V_\delta^M = \ell_\delta$.

   We also define a norm for each unitary irreducible representation of $K$. Let $\Theta$ be the restriction of the Cartan-Killing form $\mf B$ to $\mathfrak k \times \mathfrak k$. Let $\mc K_1,...,\mc K_r$ be a basis for
$\mathfrak k$ over $\R$ orthonormal with respect to $\Theta$. Let $\omega_{\mathfrak k} = -(\mc K_1^2+...+\mc K_r^2)$
be the Casimir element of $K$. Clearly $\omega_{\mathfrak k}$ is a differential operator which commutes with both left and right
translations of $K$. Thus $\delta(\omega_{\mathfrak k})$ commutes with $\delta (k)$ for all $k \in K$. Hence by Schur's lemma  \cite[Ch.I, Theorem 2.1]{Sugiura-book}: $\delta(\omega_{\mathfrak k}) = c(\delta) \delta(e)$ where $c(\delta) \in
\C$. As $\delta(\mc K_i)$ $(1 \leq i \leq r)$ are skew-adjoint operators, $c(\delta)$ is real and $c(\delta) \geq 0$. We define
$|\delta|^2 = c(\delta)$, for $\delta \in \widehat{K}_M$. As, $\delta \in \widehat{K}_M$, $\delta(k)$ is a unitary matrix of
order $d_\delta \times d_\delta $. So  $\|\delta(k)\|_{\mathbf{2}} = \sqrt{d_\delta}$ where $\|\cdot\|_{\mathbf{2}}$ denotes the Hilbert Schmidt norm. Also, from Weyl's dimension formula we can choose an $r \in \mathbb Z^+$ and a positive constant $c$ independent of $\delta$ such that $\|\delta (k)\|_{\mathbf{2}} \leq c(1+|\delta|)^r$ for all $k \in K$. Thus,
$d_\delta  \leq c' (1+|\delta|)^{2r}$ with  $c'>0$  independent of  $\delta$.\\
For any $f \in \mc C^\infty(X)$ we put:
\begin{equation}
\label{matrix valud delta projection}
f^\delta(x) = d_\delta \int_K f(kx) \delta(k^{-1}) dk.
\end{equation}
Clearly, $f^\delta$ is a $\mc C^\infty$ map from $X$ to $Hom(V_\delta, V_\delta)$ satisfying
\begin{equation}
\label{eq:pri:property of f delta}
f^\delta(kx) = \delta(k) f^\delta(x), ~~\mbox{for all~} x \in X, k \in K.
\end{equation}
Any function satisfying the property (\ref{eq:pri:property of f delta}) will be referred to as (a $d_\delta\times d_\delta$ matrix valued) left $\delta$-type function. For any function space $\mc E(X) \subseteq \mc C^\infty(X)$, we write $\mc E_\delta(X)= \{f^\delta~|~f \in \mc E(X) \}$.  We shall denote by $\check{\delta}$ the contragradient representation of the representation $\delta \in \what{K}_M$.
and a function $f$ will be called a scalar valued left $\check{\delta}$-type function if $ f \equiv d_\delta \chi_\delta \ast f$, where the operation $\ast$ is the convolution over $K$. For any class of scalar valued  functions $\mc G(X)$ we shall denote
$$\mc G(\check{\delta}, X) = \{g \in \mc G(X)~|~g \equiv d_\delta \chi_\delta \ast g \}.$$
 The following theorem, due to Helgason,  identifies the two classes $\mc D_\delta(X)$ and $\mc D(\check{\delta}, X)$ corresponding to each $\delta \in \what{K}_M $.
\begin{Proposition}
\label{left delta left delta check identification}
\emph{[Helgason \cite[Ch.III, Proposition 5.10]{Helgason-gas}]}\\
The map $\mc Q: f \mapsto g$,  $g(x)=tr\lt(f(x)\rt)$  $(x \in X)$ is a homeomorphism from $\mc D_\delta(X)$ onto $\mc D(\check{\delta}, X)$ and its inverse is given by $g \mapsto f=g^\delta$.
\end{Proposition}
\begin{Remark}
\label{rem:pri:topology of D delta X}
For each $\delta \in \what{K}_M$, the space $\mc D(X, Hom(V_\delta, V_\delta))$ of $\mc C^\infty$ functions on $X$ taking values in $Hom(V_\delta, V_\delta)$, carries the inductive limit topology of the Fr\'{e}chet spaces
$$\mc D^R(X, Hom(V_\delta, V_\delta))= \{F\in \mc D(X, Hom(V_\delta, V_\delta))~|~suppF \subseteq \overline{B^R(0)} \},$$ for $R= 0, 1,2, \cdots.$ As $\mc D(\check{\delta}, X) \subset \mc D(X)$, so the natural topology of $\mc D(\check{\delta}, X)$ is the inherited subspace topology.
\end{Remark}
A consequence of the  Peter-Weyl theorem can be stated  \cite[Ch.IV, Corollary 3.4]{Helgason-gga} in the form that  any $f \in \mc C^\infty(X)$ has the decomposition
\begin{equation}
\label{eq:pri:Peter-Weyl decomposition}
f(x)= \sum_{\delta \in \what{K}_M} tr(f^\delta(x)).
\end{equation}
A function $f \in \mc C^\infty(X)$ is said to be `left-$K$ finite' if there exists a finite subset $\Gamma(f) \subset \what{K}_M$ (depending on the function $f$)
such that  $tr(f^\gamma(\cdot)) \equiv 0$ for all $\gamma \in \what{K}_M \setminus \Gamma(f) $. For any class $\mf H(X) \subseteq \mc C^\infty(X) $ of function we shall denote $\mf H(X)_K$ for its left $K$ finite subclass. Let $\Gamma$ be a fixed subset (finite or infinite) of $\what{K}_M$. Then we shall use the notation $\mf H_\Gamma(X)$ for the  subclass of $\mf H(X)$
\begin{equation}
\label{eq:pri:finite left K type defn}
\mf H_\Gamma(X)= \{g \in \mf H(X)~|~ g^\delta(\cdot) \equiv 0, \mbox{~for all~} \delta \in \what{K}_M \setminus \Gamma \}.
\end{equation}
For each $f \in \mathcal D(X)$ and $\delta \in \widehat{K}_M$, we define the $\delta$ projection of it's HFT $\mathcal F f$ as follows:
\begin{equation}
\label{eq:delta projec of HFT}
\left(\mathcal F f \right)^\delta (\lambda, k)= d_\delta \int_K \mathcal F f(\lambda, k_1 k ) \delta(k_1^{-1}) dk_1, ~~\mbox{where}~ \lambda \in \mathfrak a^*_\C ~\mbox{and}~ k \in K.
\end{equation}
The HFT $\mathcal F( f^\delta)$ of $f^\delta$ is also defined by the formula (\ref{eq:HFT definition}), in this case the integration is taken over each matrix entry.
\begin{Lemma}
\label{Lem:2}
  For each $f \in \mathcal D(X)$ and $\delta \in \widehat{K}_M$ the following are true
\begin{enumerate}
   \item $\left( \mathcal F f \right)^\delta (\lambda, k) = \delta(k) \left(\mathcal F f \right)^\delta (\lambda, e),$
  \item $\mathcal F(f^\delta)(\lambda, k)= \left(\mathcal F f \right)^\delta (\lambda, k), \mbox{~for all}~\lambda \in \mathfrak a^*_\C~\mbox{and}~k \in K.$
\end{enumerate}
\end{Lemma}
\begin{proof}
 Part (i) of the Lemma follows trivially from (\ref{eq:delta projec of HFT}). Part (ii) can be deduced from the following
\begin{eqnarray} \label{euarray:1}
\mathcal F (f^\delta) (\lambda , kM) &=& \int_X f^\delta (x) e^{(i
\lambda - 1) \mc H(x^{-1}k)} dx , \nonumber\\
&=& d_\delta  \int_X \left\{\int_K f(k_1 x) \delta(k_1^{-1}) dk_1\right\}
e^{(i
\lambda - 1) \mc H(x^{-1}k)} dx
\end{eqnarray}
Now the desired result follows from (\ref{euarray:1}) by a simple application of the Fubini's theorem.
\end{proof}
 \section{\textbf{ The $\delta$-spherical transform}}
 \label{sec:delta spherical transofrm}
\setcounter{equation}{0}
Let us now define the `$\delta$-spherical transform' on $\mc D_\delta(X)$. Most of the basic analysis was done by Helgason \cite{Helgason-gas} on $\mathcal D({\check{\delta}}, X)$, we shall follow those results closely and prove them on $\mc D_\delta(X)$ using the homeomorphism $\mc Q$, defined in Proposition \ref{left delta left delta check identification}.
\begin{Definition}
\label{Def:delta spherical transform}
  For $f \in \mc D_\delta(X)$ the $\delta$-spherical transform $\widetilde{f}$ is an  operator valued function on $\mathfrak a^*_\C$ and is given by
\begin{equation}
\label{eq:def. delta sp transform}
\widetilde{f}(\lambda)= d_\delta \int_G trf(x) \Phi_{\overline{\lambda}, \delta}^*(x) dx
\end{equation}
where for each $\delta \in \what{K}_M$ and $\lambda \in \mf a^*_\C$, the function
\begin{equation}
\label{eq:pri:gen. sp. funct.}
\Phi_{\lambda, \delta}(x) = \int_K e^{-(i \lambda+ 1)H(x^{-1}k)} \delta(k) dk, \hspace{.2in}  x\in G,
\end{equation}
is called the `generalized spherical function' of class $\delta$. For each $x \in G$, $\Phi_{\lambda, \delta}(x)$  is an operator in $Hom(V_\delta, V_\delta)$. Taking point-wise adjoint leads to the expression
\begin{equation}
\label{eq:pri:adjoint of gen. sp. funct.}
\Phi_{\overline{\lambda}, \delta}^*(x):=\Phi_{\overline{\lambda}, \delta}^*(x) = \int_K e^{(i \lambda -1)H(x^{-1}k)} \delta(k^{-1}) dk, \hspace{.2in}  x\in G.
\end{equation}
\end{Definition}
\begin{Remark}
\label{rem:pri:gen. sp. funct.as funct.on X}
From the  Iwasawa decomposition, if $x \in G$ and $\tau \in K$, $\mc H(\tau x)= \mc H(x)$. Hence, the expressions (\ref{eq:pri:gen. sp. funct.}) and (\ref{eq:pri:adjoint of gen. sp. funct.}) show that both $\Phi_{\lambda, \delta}$ and $\Phi_{\overline{\lambda}, \delta}^*$ can be considered as functions on the  space $X=G/K$.
\end{Remark}
In the following proposition we list out some basic properties of the generalized spherical functions that we will be using.
\begin{Proposition}
\label{prop:pri:basic properties of gen, sp. funct.}
\begin{enumerate}
\item For each $x \in X$, the function $\lambda \mapsto \Phi_{\lambda, \delta}(x)$ is holomorphic on $\mf a^*_\C$.
\item Let $\delta \in \what{K}_M$ and $\lambda \in \mf a^*_\C$. Then for each $x \in X$ and $k \in K$ we have
\begin{equation}
\label{eq:pri:behaviour of gen. sp. funct. under K action}
\Phi_{{\lambda}, \delta}(kx) =
\delta(k) \Phi_{{\lambda}, \delta}(x) \mbox{~~and~~} \Phi_{\overline{\lambda}, \delta}^*(kx) = \Phi_{\overline{\lambda}, \delta}^*(x)
\delta(k^{-1}).
\end{equation}
Let $v \in V_\delta$ and $ m \in M$ then
\begin{equation}
\label{eq:pri:M fixedness of adj. of gen. sp. funct.}
\delta(m)
\left( \Phi_{\overline{\lambda}, \delta}^*(x) v\right)=
\Phi_{\overline{\lambda}, \delta}^*(x) v.
\end{equation}
\item \emph{[Helgason \cite[Ch.III, Theorem 5.15 ]{Helgason-gas}]}
For each $\delta \in \what{K}_M$, $\omega \in W$ and for all $\lambda \in \mf a^*_\C$, the restrictions $\Phi_{\lambda, \delta}|_A$ and $\Phi_{\overline{\lambda}, \delta}^*|_A$  satisfy the relations
\begin{align}
\label{W action on Phi|A}
\Phi_{\lambda, \delta}|_A Q^\delta(\lambda)&= \Phi_{\omega\lambda, \delta}|_A Q^\delta(\omega\lambda),\\
\label{W action on Phi|A*}
Q^\delta(\lambda)^{-1} \Phi_{\overline{\lambda},\delta}^*|_{A} &= Q^\delta(\omega\lambda)^{-1} \Phi_{-\overline{\omega\lambda}, \delta}^*|_A
\end{align}
 where $Q^\delta(\lambda)$ is a $(\ell_\delta \times \ell_\delta)$ matrix whose entries are certain constant coefficient polynomials in  $\lambda \in \mf a^*_\C$ (see \cite[Ch. III, \S2]{Helgason-gas} for details ). Furthermore, both sides of (\ref{W action on Phi|A*}) are holomorphic for all $\lambda \in \C$, implying that  $\Phi_{\overline{\lambda}, \delta}^*|_A$ is divisible by $Q^\delta(\lambda)$ in the ring of entire functions.
\item For each fixed $\lambda$ and $\delta$, the function
$\Phi_{{\lambda}, \delta}(x)$ and its adjoint  are both joint eigenfunctions of all $G$-invariant differential operators of $X$. Particularly, for the Laplace-Beltrami operator  $\mathbf{L}$, the eigenvalues are as follows:
\begin{equation}
\label{eq:pri:gen.sp. funct. as eigen funct. of L}
\lt( \mathbf{L} \Phi_{\lambda, \delta}\rt)(x) = - \lt(\langle \lambda, \lambda \rangle_1 + \|\rho\|^2 \rt) \Phi_{\lambda, \delta}(x), ~~ x \in X.
\end{equation}
\item For each $\delta \in \what{K}_M$, the generalized spherical function corresponding to $\delta$ is related with the elementary spherical function by the following differential equation \cite[Ch.III, \S5, Corollary 5.17]{Helgason-gas}
    \begin{equation}
    \label{relation of gen sp fn and ele sp fn}
    \Phi_{\lambda, \delta}(gK)|_{V_\delta^M} = \left(\mathbf{D}^\delta \varphi_\lambda \right)(g) Q^\delta(\lambda)^{-1}
    \end{equation}
    where $\mathbf{D}^\delta$ is a differential operator matrix of order $(d_\delta \times \ell_\delta)$. Individual matrix entries of $\mathbf{D}^\delta$ are certain constant coefficient differential operators on $G$.
\item For any $\textbf{g}_1,\textbf{g}_2  \in \mc U(\mf g_{\C}) $ there exist constants $c=c(\textbf{g}_1,\textbf{g}_2)$, $b=b(\textbf{g}_1,\textbf{g}_2) $ and $ c_0>0 $ so that
\begin{equation}
\label{eq:pri:estimate of gen sp funct}
\| \Phi_{\lambda,\delta}(\textbf{g}_1,x,\textbf{g}_2) \|_{\mathbf{2}}
 \leq c (1+|\delta|)^b (1 + \|\lambda\|)^b \varphi_{0}(x) e^{c_0 \|\Im{\lambda}\|(1+ |x|)},
\end{equation}
for all $x \in X$ and $\lambda \in \mf a^*_\C$.
 \end{enumerate}
\end{Proposition}
\begin{proof}
Property (\ref{eq:pri:behaviour of gen. sp. funct. under K action}) follows trivially from the definition of the generalized spherical function.
(\ref{eq:pri:M fixedness of adj. of gen. sp. funct.})  also follows from (\ref{eq:pri:gen. sp. funct.}) as below:
\begin{align}
\delta(m)\lt(\Phi_{\overline{\lambda}, \delta}^*(x) v \rt)&= \lt\{\int_K e^{(i \lambda- 1)\mc H(x^{-1}k)} \delta(mk^{-1}) dk \rt\}v\nonumber\\
& = \lt\{ \int_K e^{(i \lambda- 1) \mc H(x^{-1}k'm) } \delta(k'^{-1}) dk'\rt\} v. \nonumber
\end{align}
The last line follows by a simple change of variable $mk^{-1}$ to $k'^{-1}$. In the last expression above, let $x^{-1}k' = \mc K(x^{-1}k') (\exp{\mc H(x^{-1}k')}) n'$ for some $n' \in N$. As $M$ normalizes $N$ and centralizes $A$ we have $$x^{-1}k'm= \mc K(x^{-1}k')m (\exp{\mc H(x^{-1}k')}) \mc N(x^{-1}k').$$ This shows that $\mc H(x^{-1}k')= \mc H(x^{-1}k'm)$. Thus
\begin{align}
\delta(m)\lt(\Phi_{\overline{\lambda}, \delta}^*(x) v \rt)&=\lt\{ \int_K e^{(i \lambda- 1)\mc H(x^{-1}k') } \delta(k'^{-1}) dk'\rt\} v = \Phi_{\overline{\lambda}, \delta}^*(x) v.\nonumber
\end{align}
A proof of property (ii) may be found in
\cite[Ch.III, \S1 (6)]{Helgason-gas} and \cite[Ch.II, Corollary 5.20]{Helgason-gga}.
The estimate (\ref{eq:pri:estimate of gen sp funct}) is a work of Arthur \cite{Arthur79}.
\end{proof}
\begin{Remark}
\label{rem:pri:matrix of Phi lambda delta}
The property (\ref{eq:pri:M fixedness of adj. of gen. sp. funct.})
clearly shows that for each $x \in X$ the operator $\Phi_{\overline{\lambda},\delta}^*(x)$ maps $V_\delta$ to $V_\delta^M$. Hence  $\Phi_{\overline{\lambda},
 \delta}^*(x)$ is a $d_\delta \times d_\delta$ matrix whose only the first $\ell_\delta$ rows can nonzero. Consequently, for each $x \in X $, $\Phi_{\lambda, \delta}(x)$ is a $d_\delta \times d_\delta$ matrix of which only the first $\ell_\delta$ columns
 can be nonzero. In other words, the operator $\Phi_{\lambda, \delta}(x)$ vanishes identically on the orthogonal complement of the subspace $V_\delta^M$.
\end{Remark}
\begin{Remark}
\label{rem: remark about the zeros of Kostant matrix}
 In the case of the rank-one symmetric spaces the Kostant matrix $Q^\delta$ reduces to a constant coefficient polynomial (Kostant polynomial) \cite[Theorem 11.2,S11, Ch.III]{Helgason-gas}. The degree of the Kostant polynomials depends on the choice of $\delta \in \what{K}_M $. Furthermore all the zeros of the Kostant polynomials lie on the open lower half of the imaginary axis. \\
The general rank analogue of the above result states that $\det Q^\delta(\lambda) \neq 0$ for all $\lambda \in \mf a^* + i \overline{\mf a^{*+}}$. This is an easy consequence of Lemma 2.11 and Proposition 4.1 of \cite[Ch.III]{Helgason-gas}.
\end{Remark}
\begin{Lemma}
\label{Lem:Connection betwn HFT and delta sp transform}
If $ f \in \mc D_\delta(X)$, where $\delta \in \widehat{K}_M$, then $\mathcal F f (\lambda, e) = \widetilde{f}(\lambda)$ for all $\lambda \in \mathfrak a^*_\C$.
\end{Lemma}
\begin{proof}
  For any $f \in \mc D_\delta(X) $, using the topological isomorphism $\mc Q$ as described in Proposition \ref{left delta left delta check identification}, we get $tr f(\cdot) \in \mathcal D(\check{\delta}, X)$ and also $f(x)= d_\delta \int_K trf(kx) \delta(k^{-1}) dk$. Now from the definition of HFT (\ref{eq:HFT definition}) we get:
\begin{align}
\label{ali:20}
\mathcal F f(\lambda, e) & = \int_G f(x) e^{(i \lambda - \rho)\mc H(x^{-1}e)} dx,\nonumber\\
& = d_\delta \int_G \int_K trf(kx) \delta(k^{-1}) dk e^{(i \lambda - \rho)\mc H(x^{-1}e)} dx.
\end{align}
A simple application of the Fubini theorem and a substitution $kx = y$ in the integrand of (\ref{ali:20}) gives the following.   \begin{align*}
\mathcal F f(\lambda, e)&= d_\delta \int_G trf(y) \left\{  \int_K e^{(i \lambda - \rho)\mc H(y^{-1}k)} \delta(k^{-1}) dk \right\} dy,\\
&= d_\delta \int_G  trf(y) \Phi_{\overline{\lambda}, \delta}^*(y) dy
= \widetilde{f}(\lambda).
\end{align*}
\end{proof}
\begin{Lemma}
\label{Lem:inversion delta sp. transform}
  Let $f \in \mc D_\delta(X)$, then the inversion formula for the $\delta$-spherical transform (Definition \ref{Def:delta spherical transform}) is given by:
\begin{equation}
\label{eq:inversion formuta delta sp. transform}
f(x)= \frac{1}{|W|} \int_{\mathfrak a^*} \Phi_{\lambda, \delta}(x) \widetilde{f}(\lambda) |\mathbf{c}(\lambda)|^{-2} d \lambda.
\end{equation}
Furthermore we get:
$\int_G {\|f(x)\|_{\mathbf{2}}}^2 dx = \frac{1}{|W|} \int_{\mathfrak a^*} {\|\widetilde{f}(\lambda)\|_{\mathbf{2}}}^2 ~|\mathbf{c}(\lambda)|^{-2} d \lambda.$
\end{Lemma}
\begin{proof}
  The formula (\ref{eq:inversion formuta delta sp. transform}) is derived from the inversion formula (\ref{eq:HFT inversion}) of the HFT.
\begin{align*}
f(x)& = \frac{1}{|W|} \int_{\mathfrak a^*} \int_K \mathcal F f(\lambda, k) e^{-(i \lambda + \rho)(\mc H(x^{-1}k))} |\mathbf{c}(\lambda)|^{-2} dk~d \lambda ,\\
&=  \frac{1}{|W|} \int_{\mathfrak a^*} \int_K \delta(k) \mathcal F f(\lambda, e) e^{-(i \lambda + \rho)(\mc H(x^{-1}k))} |\mathbf{c}(\lambda)|^{-2} dk~d \lambda,\\
&= \frac{1}{|W|} \int_{\mathfrak a^*}  \left\{\int_K  e^{-(i \lambda + \rho)(\mc H(x^{-1}k))} \delta(k) dk \right\} \widetilde{f}(\lambda) |\mathbf{c}(\lambda)|^{-2} d \lambda,\\
& = \frac{1}{|W|} \int_{\mathfrak a^*} \Phi_{\lambda, \delta}(x) \widetilde{f}(\lambda) |\mathbf{c}(\lambda)|^{-2} d \lambda.
\end{align*}
The second and the third line of the above deduction are consequences of Lemma \ref{Lem:2} and
Lemma \ref{Lem:Connection betwn HFT and delta sp transform} respectively.\\
The second  relation of this Lemma also follows from the Plancherel formula of the HFT (\cite{Helgason-gas}, Ch.-III, $\S$1, Theorem 1.5 ). For $f \in \mc D_\delta(X)$, the HFT is defined matrix entry wise, hence in this case the Plancherel formula is given by:
\begin{equation}
\label{eq:op valued plancherel formula for HFT}
\int_G {\|f(x)\|_{\mathbf{2}}}^2 dx = \frac{1}{|W|} \int_{\mathfrak a^*} \int_K {\|\mathcal F f(\lambda, k )\|_{\mathbf{2}}}^2 |\mathbf{c}(\lambda)|^{-2} dk ~d \lambda,
\end{equation}
which follows easily from the classical Plancherel formula given in \cite{Helgason-gas}. The required Plancherel formula for the $\delta$-spherical transform  follows from (\ref{eq:op valued plancherel formula for HFT}) by using relation (i) of Lemma \ref{Lem:2} and Schur's Orthogonality relation (\cite{Sugiura-book}, Theorem 3.2).
\end{proof}
We shall next deduce the deduce a topological Paley-Wiener (P-W) theorem for the $\delta$-spherical transform (\ref{eq:def. delta sp transform}).\\
A holomorphic function $\psi: \mathfrak a^*_\C \longrightarrow Hom(V_\delta, V_\delta^M)$ is said to be of `exponential type-$R$' ($R \in \R^+$) if  $$\sup_{\lambda \in \mathfrak a^*_\C} e^{-R \|\Im \lambda\|} (1+ \|\lambda\|)^N \|\psi(\lambda)\|_{\mathbf{2}} < +\infty \mbox{\hspace{.5in}for each~ }N \in \Z^+.$$
Let $\mathcal H_\delta^R(\mathfrak a^*_\C)$ be the class of all $Hom(V_\delta, V_\delta^M)$ valued exponential type-$R$ functions on $\mf a^*_\C$ and further let $\mathcal H_\delta(\mathfrak a^*_\C) = \bigcup_{R > 0} \mathcal H^R_\delta(\mathfrak a^*_\C)$.
\begin{Theorem}
\label{The:Topological PW Theorem}
  For each fixed $\delta \in \widehat{K}_M$, the $\delta$-spherical transform given by (\ref{eq:def. delta sp transform}) is a homeomorphism between the spaces $\mc D_\delta(X)$ and $\mathcal P^\delta(\mathfrak a^*_\C)$, where
\begin{equation}
\label{eq:PW space}
\mathcal P^\delta(\mathfrak a^*_\C)= \left\{ \xi\in \mathcal H_\delta(\mathfrak a^*_\C)~|~\lambda \mapsto Q^\delta(\lambda)^{-1} \xi(\lambda) ~\mbox{is an}~ W\mbox{-invariant entire function} \right\}.
\end{equation}
Here $Q^\delta(\lambda)$ is the matrix of constant coefficient polynomials appeared in the expressions (\ref{W action on Phi|A}) and (\ref{W action on Phi|A*}).
\end{Theorem}
\begin{proof}
  Our proof solely relies on the proof of the topological Paley-Wiener theorem given by Helgason (\cite{Helgason-gas}, Ch.-III, Theorem 5.11), where he characterized the image of the space $\mathcal D(\check{\delta}, X)$ under the transform $f \mapsto \widehat{f}$, where
\begin{equation}
\label{eq:delta sp. transform Helgason}
\widehat{f}(\lambda) = d_\delta \int_G f(x)~\Phi_{\overline{\lambda}, \delta}^*(x) dx, ~~~ (\lambda \in \mathfrak a^*_\C).
\end{equation}
Helgason proved that the above transform is a topological isomorphism between the spaces $\mathcal D(\check{\delta}, X)$ and $\mathcal P^\delta(\mathfrak a^*_\C)$. From the Proposition \ref{left delta left delta check identification} and the definition (\ref{eq:def. delta sp transform}) of the $\delta$-spherical transform we get: for each $f \in \mc D_\delta(X)$, $\what{(\mc Q f)} (\lambda) = \widetilde{f}(\lambda)$,~ $(\forall \lambda \in \mf a^*_\C)$.
The proposition now follows from the fact that both the maps $\mc Q$ and $f \mapsto \widehat{f}$ are homeomorphisms.
\end{proof}
Let us now consider the function space $\mathcal P_0^\delta(\mathfrak a^*_\C)= \left\{ h \in \mathcal H(\mf a^*_\C) |~ h ~\mbox{is }~ W\mbox{-invariant} \right\}$ with the relative topology. A function $h \in \mathcal P^\delta_0(\mathfrak a^*_\C)$ can be written as $h \equiv \left(h_{ij} \right)_{{\ell_\delta} \times d_\delta}$ where each of the scalar valued component function $h_{ij}$ is entire, $W$-invariant and of exponential type. Let $\mathcal D( G// K)$ and $\mathcal D( G// K,Hom(V_\delta, V_\delta^M))$ are respectively be the spaces of scalar valued and $Hom(V_\delta, V_\delta^M)$ valued  bi-$K$-invariant, compactly supported, $\mc C^\infty$ functions on $G$. The Paley-Wiener theorem for the bi-$K$-invariant functions gives an unique $f_{ij} \in \mathcal D(G// K)$ such that $\mathcal S f_{ij} = h_{ij}$. We set $f \equiv \left( f_{ij}\right)_{{\ell_\delta} \times d_\delta}  \in \mc D(G//K, Hom(V_\delta, V_\delta^M))$ and define the spherical transform $\mc S f$ matrix entry wise to get $\mathcal S f = h$. Moreover by suitably modifying the proof of the P-W theorem for the bi-$K$-invariant functions one can show that $\mathcal S$ is infact a homeomorphism between the spaces $\mathcal D( G//K, Hom(V_\delta, V_\delta^M) )$ and $\mathcal P^\delta_0(\mathfrak a^*_\C)$. The following Lemma identifies the two PW-spaces $\mathcal P^\delta(\mathfrak a^*_\C) $ and $\mathcal P^\delta_0(\mathfrak a^*_\C)$.
\begin{Lemma}
\emph{[Helgason \cite[ Ch.-III, $\S$5, Lemma 5.12]{Helgason-gas}]}\\
\label{Lem:4}
  The mapping $\psi(\lambda) \mapsto Q^\delta(\lambda) \psi(\lambda)$ ($\lambda \in \mathfrak a^*_\C$) is a homeomorphism from $\mathcal P^\delta_0(\mathfrak a^*_\C)$ onto $\mathcal P^\delta(\mathfrak a^*_\C)$.
\end{Lemma}
\begin{Lemma}
\label{Lem:5}
  Any function $f \in \mc D_\delta(X)$ can be written as $f(x)= \mathbf{D}^\delta \phi(x)$ ($\forall x \in G$), where $\phi \in \mathcal D (G//K, Hom(V_\delta, V_\delta^M))$ and $\mathbf{D}^\delta$ is the ${d_\delta} \times {\ell_\delta}$ matrix of constant coefficient differential operators as mentioned in (\ref{relation of gen sp fn and ele sp fn}).
\end{Lemma}
\begin{proof}
  Let $f \in \mc D_\delta(X)$, then by Theorem \ref{The:Topological PW Theorem}, its $\delta$-spherical transform $\widetilde{f} \in \mathcal P^\delta(\mathfrak a^*_\C) $. Using the homeomorphism given in Lemma \ref{Lem:4}, we get an unique function $\lambda \mapsto \Phi(\lambda)= Q^\delta(\lambda)^{-1} \widetilde{f} (\lambda)$ in $\mathcal P^\delta_0(\mathfrak a^*_\C)$. By the PW-theorem for the bi-$K$-invariant functions we get a function $\phi \in \mathcal D(G//K, Hom(V_\delta, V_\delta^M))$ such that:
\begin{equation}
\label{eq:1}
\phi(x) = \frac{1}{|W|} \int_{\mathfrak a^*} \varphi_\lambda(x) \Phi(\lambda) |\mathbf{c}(\lambda)|^{-2} d \lambda.
\end{equation}
Now by applying the differential operator $\mathbf{D}^\delta$ on the both sides of (\ref{eq:1}) and by using  the absolute convergence of the integral in (\ref{eq:1}) we get:
\begin{align*}
\left(\mathbf{D}^\delta \phi\right)(x) &= \frac{1}{|W|} \int_{\mathfrak a^*} \left(\mathbf{D}^\delta \varphi_\lambda(x) \right) \Phi(\lambda) |\mathbf{c}(\lambda)|^{-2} d \lambda, \\
&= \frac{1}{|W|} \int_{\mathfrak a^*} \Phi_{\lambda, \delta}(x) Q^\delta(\lambda) \Phi(\lambda) |\mathbf{c}(\lambda)|^{-2} d \lambda,\\
&= \frac{1}{|W|} \int_{\mathfrak a^*} \Phi_{\lambda, \delta}(x) \widetilde{f}(\lambda) |\mathbf{c}(\lambda)|^{-2} d \lambda
= f(x).
\end{align*}
The second line of the above calculation is a consequence of (\ref{relation of gen sp fn and ele sp fn}) and the last line is merely the inversion formula (\ref{eq:inversion formuta delta sp. transform}).
\end{proof}
We conclude this section with a `product formula', due to Helgason, for the polynomial $\det Q^\delta(\lambda)$ with $\lambda \in \mf a^*_\C$. First we shall characterize the set $\what{K_\beta}_{M_\beta}$ ( $\beta \in \Sigma^+_0$) of equivalence classes of representations of $K_\beta$ as certain restrictions of the representations $\delta \in \what{K}_M$. Let  $\delta \in \what{K}_M$ and $ V_\delta$ and $V_\delta^M$ be as explained earlier. Let $V$ denote the $K_\beta M$-invariant subspace of $V_\delta$ generated by $V_\delta^M$. Then $V$ decomposes into $K_\beta M$-irreducible subspaces as $V= \bigoplus_{i=1}^{\ell_\delta} V_i.$
Let $\delta(i, \beta)$ be the representation of $K_\beta M$ on $V_i$ given by $\delta$. Then each $\delta(i, \beta) $ is irreducible except when $dim K_\beta =1$ and $m k m^{-1} = k^{-1}$ for all $k \in K_\beta$ and some  $m \in M$ \cite[Ch. III, Lemma 3.9 - 3.11 ]{Helgason-gas}. In this case $\delta(i, \beta)$ brakes up into two irreducible one dimensional representations as
$\delta(i, \beta) = \delta(i, \beta)_0 \oplus \delta(i, \beta)_0\check{}$
where $\delta(i, \beta)_0\check{}$ is the contragradient representation of $\delta(i, \beta)_0$ and in this particular case we choose $\delta(i, \beta)_0\check{}$ as $\delta(i, \beta)$.
\begin{Lemma}
\label{lem:delta(i, beta) for rank-one restrictions}
\emph{Helgason \cite[Ch. III, Prposition 4.3]{Helgason-gas}} For each $\beta \in \Sigma_0^+$, as $\delta $ runs through $\what{K}_M$, the representations $\delta(i, \beta) $ ($1 \leq i \leq \ell$) runs through all of $\what{K_\beta}_{M_\beta}$.
\end{Lemma}
Let $Q^{\delta(i, \beta)}_\beta (\lambda_\beta)$ be the Kostant polynomial for the rank-one symmetric space $G_\beta/K_\beta$ corresponding to the representation $\delta(i, \beta) \in \what{K_\beta}_{M_\beta}$. Then the determinant of the Kostant matrix $Q^\delta(\lambda)$ ($\lambda \in \mf a^*_\C$) can be represented by a product formula \cite[Ch. III, \S3, (50) and \S4, Theorem 4.2 ]{Helgason-gas}
\begin{equation}
\label{eq:product formula for the det of the Kostant matrix}
\det Q^\delta(\lambda) = \mc C_\delta \prod_{\beta \in \Sigma_0^+,~1 \leq i \leq \ell_\delta } Q^{\delta(i, \beta)}_\beta(\lambda_\beta), \mbox{~for all~} \lambda \in \mf a^*_\C
\end{equation}
where $\mc C_\delta$ is a nonzero constant depending on $\delta \in \what{K}_M$.
\section{\textbf{$L^p$-Schwartz spaces}}
\label{sec:Schwartz spaces}
\setcounter{equation}{0}
\begin{Definition}\emph{[Classical $L^p$ Schwartz space]}\\
\label{def:Classical L-p Schwartz space}
  A $\mc C^\infty$ function $f$ on $X$ is said to be in the $L^p$-Schwartz space ( $0< p \leq 2$ ) $\mc S^p(X)$ if for each nonnegative integer $n$ and $D, E \in \mathcal U(\mathfrak g_\C)$ the function $f$ satisfies the following decay condition:
\begin{equation}
\label{eq:Classical Sw-space decay}
  \mu_{D, E, n}(f)= \sup_{x \in G} |f(D, x, E)| ~\varphi_0^{-\frac{2}{p}}(x)~ (1 + |x|)^n~ < +\infty.
\end{equation}
\end{Definition}
The topology induced by the countable family of seminorms $\mu_{D, E, n}(\cdot)$ makes $\mc S^p(X) $ a Fr\'{e}chet space. It can be shown that $\mathcal D(X)$ is a dense subspace of $\mc S^p(X)$ for $0 < p \leq 2$. Let $f \in \mc S^p(X)$, we take its $\delta$-projection $f^\delta$ as defined in (\ref{matrix valud delta projection}). Then $f^\delta$ is a left $\delta$-type $Hom(V_\delta, V_\delta)$ valued function  with a decay
 \begin{equation}
\label{eq:op-valued delta Sw-space decay}
\mu_{D, E, n}(f^\delta)= \sup_{x \in G} \|f^\delta(D, x, E)\|_{\mathbf{2}}~ \varphi_0^{-\frac{2}{p}}(x)~ (1 + |x|)^n~ < +\infty,
\end{equation}
where $\|\cdot\|_{\mathbf{2}}$ denotes the Hilbert Schmidt norm. Let us denote $\mc S^p_\delta(X)$ for the class of left $\delta$-type $Hom(V_\delta , V_\delta)$ valued functions with the decay (\ref{eq:op-valued delta Sw-space decay}). For each $h \in \mc S^p_\delta(X)$, the scalar valued function $tr h (\cdot)$ satisfies $tr h \equiv d_{\check{\delta}}( \chi_{\check{\delta}} \ast tr h)$ and the decay (\ref{eq:Classical Sw-space decay}). Let $\mc S^p(\check{\delta},X) \subset \mc S^p(X)$ be the class of all scalar valued left-$\check{\delta}$ type Schwartz class functions. Both the spaces
$\mc S^p_\delta(X)$ and $\mc S^p(\check{\delta},X)$ becomes Fr\`{e}chet spaces with the topologies induced by the family of seminorms given in (\ref{eq:op-valued delta Sw-space decay}) and (\ref{eq:Classical Sw-space decay}) respectively. Moreover, $\mc D_\delta(X)$ and $\mathcal D (\check{\delta},X)$ are respectively dense subspaces of $\mc S^p_\delta(X)$ and $\mc S^p(\check{\delta},X)$ in the respective Schwartz space topologies. The topological isomorphism $\mc Q$ in Proposition \ref{left delta left delta check identification} can be extended between the Schwartz spaces $\mc S^p_\delta(X)$ and $\mc S^p(\check{\delta},X)$.

Next we shall to extend the definition of the $\delta$-spherical transform to the Schwartz space $\mc S^p_\delta(X)$ ($0 < p \leq 2$). The spherical transform (\ref{eq:spherical transform}) defined on $\mathcal D( G//K)$ can be extended to the $L^p$-Schwartz spaces $\mc S^p(G//K)$ of bi-$K$-invariant functions on the group $G$.
The image $\mc S(\mf a^*_\e)$ (defined below) of $\mc S^p(G//K)$ under the spherical transform is again a Schwartz class of functions defined o the complex tube  $\mf a^*_\e = \mf a^* + i C^{\e \rho}$ where $\e = \left(\frac{2}{p} - 1 \right)$ and
$$C^{\varepsilon \rho}= \{\lambda \in \mathfrak a^*~|~ \omega \lambda(H) \leq \varepsilon \rho(H) \mbox{~for all~} H \in \overline{\mathfrak a^+} \mbox{~and~} \omega \in W \}. $$
\begin{Definition}
\label{Def:even Sw space on lambda}
  The space $\mc S(\mathfrak a^*_\e)$ consists of the complex valued functions $h$ on $\mathfrak a^*_\e$ such that:
  \begin{enumerate}
    \item $h$ is holomorphic in the interior of the tube $\mathfrak a^*_\e$ and continues on the closed tube,
    \item  $h $ is $W$-invariant,
    \item for any polynomial $P$ in the algebra  $ S(\mf a)$ of the symmetric polynomials on $\mf a^*$ and any positive integer $r$
\begin{equation}
\label{eq:Sw. space decay on lambda}
\tau_{P,r}(h)=\sup_{\lambda \in Int \mathfrak a^*_\e} (1 + \|\lambda\|)^r \left| P\left(\frac{\partial}{\partial \lambda}  \right) h(\lambda)\right| < + \infty.
\end{equation}
  \end{enumerate}
\end{Definition}
The countable family $\left\{\tau_{P,r} ~|~P \in S(\mf a), r \in \Z^+ \cup\{0\} \right\}$ gives a Fr\`{e}chet norm on the space $S(\mf a^*_\e)$.
\begin{Remark}
\label{rem:Alt form of seminorms on even SW space on lambda}
The topology of $\mc S(\mf a^*_\e)$ can also be given by the following equivalent family of seminorms.
\begin{equation}
\label{eq:Sw. space decay reduced on lambda}
\tau_{P, r}^+(h)= \sup_{\lambda \in Int \mathfrak a^*_\e \cap \left( \mathfrak a^* + i \overline{\mathfrak a^{*+}}\right)}(1 + \|\lambda\|)^r \left| P\left(\frac{\partial}{\partial \lambda}  \right) h(\lambda)\right| < + \infty.
\end{equation}
This is an easy consequence of (ii) of Definition \ref{Def:even Sw space on lambda} and the fact $\|\omega \lambda \| = \|\lambda\|$ for all $\lambda \in \mf a^*_\C$ and $\omega \in W$.
\end{Remark}
We now state the $L^p$-Schwartz space isomorphism theorem for the bi-$K$-invariant functions.
\begin{Theorem}
\label{Theo:SW. iso the for C valued spherical transform}
  The spherical transform $f \mapsto \mathcal S f$ defined in (\ref{eq:spherical transform}) is topological isomorphism between the spaces $\mc S^p(G //K)$ ($0 <p \leq 2$) and $\mc S(\mathfrak a^*_\e)$.
\end{Theorem}
For  a proof of this theorem one can see \cite{Gangolli88}. We are mainly interested in Anker's proof \cite{Anker91} of the above theorem, which relies on the PW theorem for the spherical transform.\\
Let $\mc S_0(\mathfrak a^*_\e)$ be the space of all $Hom(V_\delta, V_\delta^M)$ valued functions $h$ on $\mathfrak a^*_\e $ satisfying  (i), (ii) of Definition \ref{Def:even Sw space on lambda} along with the decay: for each polynomial $P \in S(\mf a)$ and integer $n \geq 0$
\begin{equation}
\label{eq:op valued Sw. space decay reduced on lambda}
\tau_{P, n}^+(h)= \sup_{\lambda \in Int \mathfrak a^*_\e \cap \left( \mathfrak a^* + i \overline{\mathfrak a^{*+}}\right)}(1 + \|\lambda\|)^n \left\| P\left(\frac{\partial}{\partial \lambda}  \right) h(\lambda)\right\|_{\mathbf{2}} < + \infty.
\end{equation}
The countable family $\{\tau_{P, n}^+\}$ induces a   a Fr\`{e}chet structure on $\mc S_0(\mathfrak a^*_\e)$.
The space $\mc S_0(\mf a^*_\e)$ can also be viewed as a space of $Hom(V_\delta, V_\delta^M)$ valued functions with each of its matrix entry function in $\mc S(\mf a^*_\e)$.
For our purpose we shall use the following  equivalent (inducing the same topology) family of seminorms on $\mc S_0(\mf a^*_\e)$:
   \begin{equation}\label{eq: alternative form of the seminotms}
    \tau^{+*}_{P, r} (h)= \hspace{-.2in}\sup_{\lambda \in Int(\mf a^*_\e \cap  (\mf a^* + i \overline{\mf a^{*+}}))} \left\| P\left(\frac{\partial}{\partial \lambda} \right)\left\{ h(\lambda) (\|\rho\|^2 + \langle \lambda, \lambda\rangle_1)^r\right\} \right\|_{\mathbf{2}},~  P \in S(\mf a), r \in \Z^+ \cup \{0 \}.
 \end{equation}
 Let $f =(f_{ij}) \in \mc S^p(G//K, Hom(V_\delta, V_\delta^M))$. Then by defining $\mc S f = (\mc S f_{ij})$ we can the
  following extension of the isomorphism of Theorem \ref{Theo:SW. iso the for C valued spherical transform}.
\begin{Lemma}
\label{Lem:SW. iso the for op. valued spherical transform}
  The spherical transform is a topological isomorphism between the spaces $\mc S^p(G // K, Hom(V_\delta,V_\delta^M))$ and $\mc S_0(\mathfrak a^*_\e)$.
\end{Lemma}
This Lemma can be proved easily by using the conclusion of the Theorem \ref{Theo:SW. iso the for C valued spherical transform} for each matrix entry of the functions of $\mc S^p( G //K, Hom(V_\delta,V_\delta^M))$.\\
We now define the ambient space for the image of the Schwartz space $\mc S^p_\delta(X)$ under the $\delta$-spherical transform.
\begin{Definition}
\label{def: delta SW space in the image side}
Let $\mc S_\delta(\mathfrak a^*_\e)$ be the set of all $Hom(V_\delta, V_\delta^M)$ valued functions $h$ on the complex tube $\mathfrak a^*_\e$ such that
\begin{enumerate}
  \item h is holomorphic in the interior of the tube $\mathfrak a^*_\e$ and extends as a continuous function to the closed tube.
  \item  $\lambda \mapsto Q^\delta(\lambda)^{-1}~h(\lambda)$ is $W$-invariant holomorphic function the interior of the complex tube $\mf a^*_\e$.
  \item for each polynomial $P \in S(\mf a)$ and integer $n \geq 0$
\begin{equation}
\label{eq:seminorm on S delta imageside}
\nu_{P, n} (h) =\sup_{\lambda \in Int \mathfrak a^*_\e} (1 + \|\lambda\|) \left\| P\left(\frac{\partial}{\partial \lambda}\right) h(\lambda)\right\|_{\mathbf{2}} < + \infty.
\end{equation}
\end{enumerate}
\end{Definition}
Clearly, $\mc S_\delta(\mf a^*_\e) $ becomes a Fr\`{e}chet space with the topology induced by the countable family $\left\{\nu_{P, n} \right\}$ of seminorms.
\begin{Remark}
  It is easy to  observe that  the topology of the space $\mc S_\delta(\mf a^*_\e)$  can also be determined by the following equivalent family of  seminorms: for each polynomial $P \in S(\mf a)$ and each nonnegative integer $n$
\begin{equation}
\label{eq:reduced seminorm on S delta imageside}
h \mapsto {\nu}^{\#}_{P, n} (h) =\sup_{\lambda \in Int \mathfrak a^*_\e} \left|P\left( \frac{\partial}{\partial \lambda}\right) \left\{ (\langle\lambda, \lambda \rangle_1 + \|\rho\|^2 )^n h(\lambda) \right\} \right|.
\end{equation}
\end{Remark}
Let us now state the first main theorem of this paper.
\begin{Theorem}
\label{Theorem:main-1}
The $\delta$-spherical transform (\ref{eq:def. delta sp transform}) is a topological isomorphism between the Schwartz spaces $\mc S^p_\delta(X)$ and $\mc S_\delta(\mf a^*_\e)$ with $0 < p \leq 2$ and $\e= \left(\f{2}{p}- 1 \right)$.
\end{Theorem}

The following proposition is a key step to prove Theorem \ref{Theorem:main-1}.
\begin{Proposition}
\label{Proposition:Lower bound for the Kostants polynomials}
For each $\delta \in \what{K}_M$, there exists a  $\delta$-dependent constant $c_\delta>0$ such that
\begin{equation}
\label{eq:Lower bound for the Kostants polynomials}
\inf_{\lambda \in (\mf a^* + i\overline{ \mf a^{*+}})} |\det Q^\delta(\lambda)| \geq c_\delta.
\end{equation}
\end{Proposition}
\begin{proof}
To prove this we use the product formula (\ref{eq:product formula for the det of the Kostant matrix}) of $\det Q^\delta(\lambda)$. Each of the factors $Q^{\delta(i, \beta)}_\beta$ is the Kostant polynomials of he rank-one restrictions $G_\beta/ K_\beta$ and hence it is a polynomial in one complex variable $\lambda_\beta \in {\mf a^*_{(\beta)}}_\C \cong \C$.\\
It is easy to check that $\lambda \in \mf a^* + i\overline{\mf a^{*+}}$ if and only if $\Re{\langle i \lambda, \alpha \rangle} \leq 0$ for all $\alpha \in \Sigma^+$. Using the definition of the restriction $\lambda_\beta  = \frac{\langle \lambda, \beta \rangle_1}{\langle \beta, \beta \rangle} \beta$  we get
\begin{align}
\label{relation of Killing form and its restriction}
\langle i \lambda_\beta, \beta \rangle_\beta & = \left\langle i \frac{\langle \lambda, \beta \rangle_1}{\langle \beta, \beta \rangle} \beta, \beta  \right\rangle_\beta
 = \langle i \lambda, \beta \rangle
\end{align}
where $\langle \cdot, \cdot \rangle_\beta$ denotes the inner product on $\mf a^*_{(\beta)}$ as well as its $\C$-bilinear extension to the complexification ${\mf a^*_{(\beta)}}_\C$. Now (\ref{relation of Killing form and its restriction}) clearly suggests that if $\lambda \in \mathfrak a^* + i \overline{\mathfrak a^{*+}}$ then for each $\beta \in \Sigma_0^+$
 the restriction $\lambda_\beta \in \mathfrak a_{(\beta)}^* + i \overline{\mathfrak a_{(\beta)}^{*+}}$. As the polynomial $Q^{\delta(i, \beta)}_\beta(\lambda_\beta ) \neq 0$ for all $\lambda_\beta \in \mathfrak a_{(\beta)}^* + i \overline{\mathfrak a_{(\beta)}^{*+}}$ and $\mathfrak a_{(\beta)}^* + i \overline{\mathfrak a_{(\beta)}^{*+}}$ being a closed subset of $\C$ so we get a positive constant $c_{\delta(i, \beta)}$ such that $$\inf_{\lambda_\beta \in \mathfrak a_{(\beta)}^* + i \overline{\mathfrak a_{(\beta)}^{*+}}}|Q^{\delta(i, \beta)}_\beta(\lambda_\beta)| \geq c_{\delta(i, \beta)}. $$
Since corresponding to each $\delta \in \what{K}_M$ the product formula (\ref{eq:product formula for the det of the Kostant matrix}) of $\det Q^\delta(\lambda)$ has only finitely many factors hence we get the desired conclusion of the proposition.
\end{proof}
The next lemma extends the homeomorphism given in Lemma \ref{Lem:4} between the P-W spaces to the corresponding Schwartz classes.
\begin{Lemma}
\label{Lem:S-delta and S-0 homeomorphic}
  The map
\begin{equation}
\label{eq:homeo betwn sw spaces}
g(\lambda) \mapsto Q^\delta(\lambda) g(\lambda),~~\mbox{~~for all ~} \lambda \in \mathfrak a^*_\e,
\end{equation}
is a homeomorphism from the space $\mc S_0(\mathfrak a^*_\e)$ onto  $\mc S_\delta(\mathfrak a^*_\e)$.
\end{Lemma}
\begin{proof}
  Let us first take $g \in \mc S_0(\mathfrak a^*_\e)$. We denote $h(\cdot) = Q^\delta(\cdot) g(\cdot)$. We shall show that $h \in \mc S_\delta(\mathfrak a^*_\e)$. As $Q^\delta(\lambda)$ is an ${\ell_\delta} \times {\ell_\delta}$ matrix of polynomials in $\lambda \in \mathfrak a^*_\C$ so  $\lambda \mapsto Q^\delta(\lambda)$ is a holomorphic function on $\mathfrak a^*_\C$.
Hence the function $h$ satisfies condition (i) and (ii) of Definition \ref{def: delta SW space in the image side} which easily follows from the similar properties of $g \in \mc S_0(\mf a_\e)$ and the construction (\ref{eq:homeo betwn sw spaces}) of the function $h$.\\
To establish the decay condition (\ref{eq:seminorm on S delta imageside}) for $h$ let us take a polynomial $P \in S(\mf a)$ and $m \in \Z^+ \cup \{ 0\}$. Then
\begin{align}
\label{ali:1}
\sup_{\lambda \in Int\mathfrak a^*_\e} &\left\|P\left(\frac{\partial}{\partial \lambda}\right) h(\lambda) \right\|_{\mathbf{2}} (1+ \|\lambda\|)^m \nonumber \\
& \leq \sup_{\lambda \in Int\mathfrak a^*_\e} \sum_\kappa c_\kappa \left\|\left\{P'_\kappa\left(\frac{\partial}{\partial \lambda}\right)Q^\delta(\lambda) \right\} \left\{P_\kappa\left(\frac{\partial}{\partial \lambda}\right)g (\lambda)\right\} \right\|_{\mathbf{2}} (1 + \|\lambda\|)^m \nonumber\\
& \leq \sup_{\lambda \in Int\mathfrak a^*_\e} \sum_{\kappa} c_\kappa \left\|\left\{ P'_\kappa \left(\frac{\partial}{\partial \lambda}\right)Q^\delta(\lambda) \right\}\right\|_{\mathbf{2}} \left\| \left\{P_\kappa \left(\frac{\partial}{\partial \lambda}\right)g (\lambda) \right\}\right\|_{\mathbf{2}} (1 + \|\lambda\|)^m \nonumber\\
& \leq \sup_{\lambda \in Int\mathfrak a^*_\e} \sum_\kappa c_\kappa^
\delta \left\| \left\{P_\kappa\left(\frac{\partial}{\partial \lambda}\right)g \right\}(\lambda) \right\|_{\mathbf{2}} (1 + \|\lambda\|)^{m_\kappa^\delta}
\end{align}
where, $m_\kappa^\delta$ are nonnegative integers and $c_\kappa^\delta$ are positive constants both depending on $\delta \in \widehat{K}_M$. As $g \in \mc S_0(\mathfrak a^*_\e)$,  the right hand side of (\ref{ali:1}) is clearly finite. Moreover, (\ref{ali:1}) shows that the map (\ref{eq:homeo betwn sw spaces}) is a continuous function from $\mc S_0(\mathfrak a^*_\e)$ into $\mc S_\delta(\mathfrak a^*_\e)$.\\
Now let $\psi \in \mc S_\delta(\mathfrak a^*_\e)$ and define $g(\cdot) := Q^\delta(\cdot)^{-1} \psi(\cdot)$.
 As, $\psi \in S_\delta(\mf a^*_\e)$, by Definition \ref{def: delta SW space in the image side} the function
 $$ \lambda \mapsto g(\lambda )= \frac{1}{\det Q^\delta(\lambda)} Q^\delta_c(\lambda) \psi(\lambda)$$
 (here, $Q^\delta_c(\lambda)$ is the cofactor matrix of $Q^\delta(\lambda)$) is $W$-invariant and it is holomorphic
in the interior of the tube $\mf a^*_\e$. To infer $g \in \mc S_0(\mf a^*_\e)$ all we have to show is that the function
$g$ has certain decay.  Let $P \in S(\mf a)$ and $t$ be any nonnegative integer, then
\begin{align}
\label{ali:imp1}
\left\{P\left(\f{\partial}{\partial \lambda}\right) g(\lambda)  \right\}  &= \left\{P\left(\f{\partial}{\partial \lambda}\right) \f{1}{\det Q^\delta(\lambda)} Q^\delta_c(\lambda) \psi(\lambda)  \right\} \nonumber\\
&= \sum_\kappa \f{P_\kappa\left(\f{\partial}{\partial \lambda}\right) Q^\delta_c(\lambda) P'_\kappa\left(\f{\partial}{\partial \lambda} \right) \psi(\lambda)}{\left(\det Q^\delta(\lambda) \right)^{m_\kappa}}.
\end{align}
The last line of (\ref{ali:imp1}) follows by an easy application of the Leibniz rule. Here $P_\kappa, P'_\kappa$ are finite degree polynomials, $m_{\kappa}$ is a positive integer depending on $\kappa$ and the sum is over a finite set. From (\ref{ali:imp1}) we get:
\begin{align}
\label{ali:imp2}
\left\|P\left(\f{\partial}{\partial \lambda}\right) g(\lambda)  \right\|_{\mathbf{2}} & \leq \sum_\kappa \f{\|P_\kappa\left(\f{\partial}{\partial \lambda}\right) Q^\delta_c(\lambda)\|_{\mathbf{2}} \|P'_\kappa\left(\f{\partial}{\partial \lambda} \right) \psi(\lambda)\|_{\mathbf{2}}}{|\det Q^\delta(\lambda) |^{m_\kappa}}  \nonumber\\
&\leq c(\delta)  \sum_\kappa \f{ \|P'_\kappa\left(\f{\partial}{\partial \lambda} \right) \psi(\lambda)\|_{\mathbf{2}}}{|\det Q^\delta(\lambda) |^{m_\kappa}} (1 + \|\lambda\|)^{n_\kappa}.
\end{align}
The above inequality is obtained by using the fact that: $\|P_\kappa\left(\f{\partial}{\partial \lambda}\right) Q^\delta_c(\lambda)\|_{\mathbf{2}} \leq c(\delta) (1 + \|\lambda \|)^{n_\kappa}$ where $c(\delta)>0$ is a $\delta$-dependent constant and $n_\kappa$ is a positive integer depending on the degree of $P_\kappa$ (it may also depend on $\delta$).
Now from (\ref{ali:imp2}) we get the following inequality for any nonnegative integer $t$.
\begin{align}
\label{ali:imp3}
\sup_{Int(\mf a^*_\e \cap (\mf a^*+ i \overline{\mf a^{*+}}))}\left\|P\left(\f{\partial}{\partial \lambda} \right) g(\lambda) \right\|_{\mathbf{2}} (1+ \|\lambda\|)^t &\leq \sum_{\kappa} \sup_{Int(\mf a^*_\e \cap (\mf a^*+ i \overline{\mf a^{*+}}))} \hspace{-.3in}\f{\left\|P'_\kappa\left(\f{\partial}{\partial \lambda} \right) \psi(\lambda) \right\|_{\mathbf{2}} (1+ \|\lambda\|)^{t+n_\kappa}}{|\det Q^\delta(\lambda)|^{m_\kappa}}\nonumber\\
&\leq \sum_{\kappa} \f{ \sup_{\lambda \in \mf a^*_\e} \left\|P'_\kappa\left(\f{\partial}{\partial \lambda} \right) \psi(\lambda) \right\|_{\mathbf{2}} (1+ \|\lambda\|)^{t+n_\kappa}}{ \inf_{\lambda \in Int(\mf a^*_\e \cap (\mf a^*+ i \overline{\mf a^{*+}}))}|\det Q^\delta(\lambda)|^{m_\kappa}}\nonumber\\
&\leq \sum_{\kappa} \f{1}{c_\delta^\kappa} \sup_{\lambda \in \mf a^*_\e} \left\|P'_\kappa\left(\f{\partial}{\partial \lambda} \right) \psi(\lambda) \right\|_{\mathbf{2}} (1+ \|\lambda\|)^{t+n_\kappa}.
\end{align}
The last line of the above successive inequalities is a consequence of Proposition \ref{Proposition:Lower bound for the Kostants polynomials}. As $\psi \in \mc S_\delta(\mf a^*_\e)$ so each term of the finite summation on the right hand side of (\ref{ali:imp3}) is finite. Hence we conclude that $g$ satisfies the decay (\ref{eq:op valued Sw. space decay reduced on lambda}) of the space $\mc S_0(\mf a^*_\e)$. The inequality  (\ref{ali:imp3}) further concludes that the continuous  map (\ref{eq:homeo betwn sw spaces}) from $\mc S_0(\mf a^*_\e)$ onto $\mc S_\delta(\mf a^*_\e)$ is injective and also its inverse map is continuous. As both the spaces $\mc S_0(\mf a^*_\e)$ and $\mc S_\delta(\mf a^*_\e)$ are Fr\`{e}chet spaces so the map (\ref{eq:homeo betwn sw spaces}) is a homeomorphism.
\end{proof}
\begin{Lemma}
\label{Lem:delta PW is dence in delta SW space}
  The Paley-Wiener space $\mathcal P^\delta(\mathfrak a^*_\C)$, defined in (\ref{eq:PW space}), is a dense subspace of the Schwartz space $\mc S_\delta(\mathfrak a^*_\e)$.
\end{Lemma}
\begin{proof}
 Let us take any $H \in \mc S_\delta(\mathfrak a^*_\e)$, it is enough to show that, there is a sequence $\{G_n\}$ ($G_n \in \mathcal P(\mathfrak a^*_\C)$) converging to $H$ in the topology of the space $\mc S_\delta(\mathfrak a^*_\e)$. Let $H(\lambda) = \left( H_{ij}(\lambda)\right)_{{\ell_\delta} \times d_\delta}$. By  the isomorphism obtained in Lemma \ref{Lem:S-delta and S-0 homeomorphic}, we get one unique $G \in \mc S_0(\mathfrak a^*_\e)$ such that
\begin{align}
\label{ali:5}
H(\lambda) = \left( H_{ij}(\lambda)\right)_{{\ell_\delta} \times d_\delta}& = Q^\delta(\lambda) G(\lambda) \nonumber,\\
&= Q^\delta(\lambda) \left( G_{ij}(\lambda)\right)_{{\ell_\delta} \times d_\delta} \nonumber,\\
&= \left( \sum_{k=1}^{{\ell_\delta}} Q^\delta(\lambda)_{ik} G_{kj}(\lambda) \right)_{{\ell_\delta} \times d_\delta}.
\end{align}
As $G \in \mc S_0(\mathfrak a^*_\e)$, so from the definition of the Schwartz space $\mc S_0(\mathfrak a^*_\e)$  it follows that the matrix entry functions $G_{ij} \in \mc S(\mathfrak a^*_\e)$ for each $1 \leq i \leq {\ell_\delta}$ and $1 \leq j \leq d_\delta$. We know that the Paley-Wiener space $\mathcal P(\mathfrak a^*_\C)$ under the spherical transform is dense in the Schwartz class $\mc S(\mathfrak a^*_\e)$ \cite{Gangolli88}. Therefore we can get a sequence $\left\{{g_{ij}}_n \right\}_n \subset \mathcal P(\mathfrak a^*_\C)$ converging to $G_{ij}$ in $\mc S(\mathfrak a^*_\e)$. As, each $Q^\delta_{ik}(\lambda)$ ($1 \leq i \leq {\ell_\delta}, 1 \leq k \leq d_\delta$) is a polynomial in $\lambda$, so the sequence $\left\{\sum_{k=1}^{\ell_\delta} Q^\delta_{ik}(\lambda) {g_{ij}}_n(\lambda) \right\}_n$ converges to $\sum_{k=1}^{\ell_\delta} Q^\delta_{ik}(\lambda)  G_{kj}(\lambda)$ (for each $\lambda$) in $\mc S(\mathfrak a^*_\e)$.\\
Let $g_n(\lambda) = \left({g_{ij}}_n(\lambda) \right)_{{\ell_\delta} \times d_\delta}$. As each ${g_{ij}}_n \in \mathcal P(\mathfrak a^*_\C)$, so  from the definition it follows that, the matrix valued function $g_n \in \mathcal P^\delta_0(\mathfrak a^*_\C) $. Clearly by Lemma \ref{Lem:4} for each natural number $ n $, $Q^\delta(\cdot) g_n(\cdot) \in \mathcal P^\delta(\mathfrak a^*_\C)$.  Let $P$ be any polynomial in $S(\mf a)$ and $t$ be any nonnegative integer then:
\begin{align}
\label{ali:6}
&\tau_{P,t}\left(Q^\delta(\cdot)g_n(\cdot) - Q^\delta(\cdot)G(\cdot) \right) \nonumber\\
&= \sup_{\lambda \in Int \mathfrak a^*_\e} \left\|P\left(\frac{\partial}{\partial \lambda}\right) \left\{ Q^\delta(\lambda) g_n(\lambda) - Q^\delta(\lambda) G(\lambda)\right\} \right\|_{\mathbf{2}} (1+ \|\lambda\|)^t \nonumber,\\
&= \sup_{\lambda \in Int \mathfrak a^*_\e} \sum_{i=1}^{\ell_\delta} \sum_{j= 1}^{d_\delta} \left\|P\left(\frac{\partial}{\partial \lambda}\right) \sum_{k=1}^{{{\ell_\delta}}} \left\{Q^\delta(\lambda)_{ik} {g_{ik}}_n(\lambda) - Q^\delta(\lambda)_{ik} G_{ik}(\lambda) \right\} \right\|^2 (1+ \|\lambda\|)^t.
\end{align}
A suitable choice of $n$ can made the right hand side of (\ref{ali:6}) arbitrarily small. Hence we get the sequence $\left\{ Q^\delta(\cdot) g_n \right\}_n$ in $\mathcal P^\delta(\mathfrak a^*_\C)$ converging to $H$ in the topology of $\mc S_\delta(\mathfrak a^*_\e)$. This completes the proof of the Lemma.
\end{proof}
Next we shall try to extend the definition  of the $\delta$-spherical transform (\ref{eq:def. delta sp transform}) to the Schwartz class $\mc S^p_\delta(X)$ where $0 < p \leq 2$.
\begin{Lemma}
\label{Lem:analytic in the interior}
  For each $f \in \mc S^p_\delta(X)$, the function $\lambda \mapsto \widetilde{f}(\lambda)$, where $\widetilde{f}$ is given by (\ref{eq:def. delta sp transform}), is a holomorphic function in the interior of the complex tube $\mathfrak a^*_\e$.
\end{Lemma}
\begin{proof}
For each $f \in \mc S^p_\delta(X)$ the function $x \mapsto trf(x)$ has the following decay: ~for each $D, E \in \mathcal U(\mathfrak g_\C)$ and integer $n \geq 0$
\begin{equation}
\label{eq:2}
\sup_{x \in G} |trf(D, x, E)| (1 + |x|)^n \varphi_0^{-\frac{2}{p}}(x) < + \infty,
\end{equation}
which follows easily from (\ref{eq:op-valued delta Sw-space decay}) and the fact that $|trf(x)| \leq \|f(x)\|_{\mathbf{2}}$ ( $x \in X$). Using (\ref{eq:2}) and the estimate (\ref{eq:pri:estimate of gen sp funct}) of the generalized spherical functions one can show, by following a standard argument (see \cite[$\S$6.2]{Gangolli88},~\cite{Eguchi76}), that the $\delta$-spherical transform, defined by the integral (\ref{eq:def. delta sp transform}), exists for $\lambda \in \mathfrak a^*_\e$. \\
Let $\gamma$ be a closed curve in the interior of the tube $\mathfrak a^*_\e$. Then for $f \in \mc S^p_\delta(X)$ we get
\begin{align*}
\int_{\gamma} \widetilde{f}(\lambda) d \lambda &= d_\delta \int_{\gamma} \left\{\int_G tr f(x) \Phi_{\overline{\lambda}, \delta} (x) dx \right\} d \lambda.
\end{align*}
As the integral within braces exists absolutely for $\lambda \in \mathfrak a^*_\e$, so we apply Fubini's theorem to get:~$\int_{\gamma} \widetilde{f}(\lambda) d \lambda
= d_\delta \int_G  tr f(x) \left\{ \int_{\gamma} \Phi_{\overline{\lambda}, \delta} (x) d \lambda \right\} dx$.
We also recall that the functions $\lambda \mapsto \Phi_{\overline{\lambda}, \delta} (\cdot)$ are entire. Hence by an application of Morera's theorem the desired conclusion of the lemma follows.
\end{proof}
\section{\textbf{ Proof of Theorem \ref{Theorem:main-1}}}
\label{sec:Proof-Theorem-main}
\setcounter{equation}{0}
To show that the $\delta$-spherical transform is a topological isomorphism it is enough to show that it is a continuous surjection from $\mc S^p_\delta(X)$ onto $\mc S_\delta(\mf a^*_\e)$.
\begin{Lemma}
\label{Lem:continuity of the delta spherical transform}
  The $\delta$-spherical transform $f \mapsto \widetilde{f}$ is a continuous map from the Schwartz space $\mc S^p_\delta(X)$ into $\mc S_\delta(\mathfrak a^*_\e)$.
\end{Lemma}
\begin{proof}
 It is enough to show that given any seminorm $\nu$ (or equivalently $\nu^{\#}$) on $\mc S_\delta(\mathfrak a^*_\e)$ one can find a seminorm $\mu$ on $\mc S^p_\delta(X)$ such that  $$\nu(\widetilde{f}) \leq c \mu(f) \mbox{~for all~} f \in \mc S^p_\delta(X).$$
With $\texttt{P} \in S(\mf a) $ and $t \in \mathbb{N} \cup \{ 0\}$ we get the following by using the integral expression (\ref{eq:def. delta sp transform}) of the $\delta$-spherical transform.

\begin{align}
\label{ali:10}
\texttt{P}\left( \frac{\partial}{\partial \lambda}\right) \left\{ \left(\langle \lambda, \lambda \rangle_1 + \|\rho\|^2\right)^t  \widetilde{f}(\lambda)\right\}&= \texttt{P}\left( \frac{\partial}{\partial \lambda}\right) \int_G trf(x) ~\left(\langle \lambda, \lambda \rangle_1 + \|\rho\|^2\right)^t \Phi_{\overline{\lambda}, \delta}^*(x) dx \nonumber\\
& = (-1)^n \texttt{P}\left( \frac{\partial}{\partial \lambda}\right) \int_G trf(x)~~ \mathbf{L}^t \Phi_{\overline{\lambda}, \delta}^*(x) dx.
\end{align}
The last equality is a consequence of the property (\ref{eq:pri:gen.sp. funct. as eigen funct. of L}) of the generalized spherical function. Now a simple application of integration by parts gives:
\begin{align}
\label{ali:101}
(\ref{ali:10}) & = (-1)^n \texttt{P}\left( \frac{\partial}{\partial \lambda}\right) \int_G \mathbf{L}^t trf(x)~~ \Phi_{\overline{\lambda}, \delta}^*(x) dx \nonumber\\
&=(-1)^n \texttt{P}\left( \frac{\partial}{\partial \lambda}\right) \int_G  tr ~\mathbf{L}^tf(x) ~~  \Phi_{\overline{\lambda}, \delta}^*(x) dx \nonumber\\
&= (-1)^n \texttt{P}\left( \frac{\partial}{\partial \lambda}\right) \int_G  tr ~\mathbf{L}^tf(x) \left\{ \int_K e^{(i \lambda - \rho)(\mc H(x^{-1}k))} \delta(k^{-1}) dk \right\} dx.
\end{align}
The second line in the above chain of equalities uses the fact $\mathbf{L}tr f(\cdot)= tr \mathbf{L}f(\cdot)$ which is clear as the differential operator $\mathbf{L}$ acts entry-wise to the operator valued function $f $.\\
As $f \in \mc S^p_\delta(X)$ so it  can also be considered as a right-$K$-invariant function on the group $G$. The action of the Laplace Beltrami operator $\mathbf{L}$ on $f$ is the same as the action of the Casimir operator on $f$ considering as a function on $G$. Therefore, by the property of the Casimir operator,  the action of $\mathbf{L}$ does not change the  left-$K$-type of the function $f$, i. e the function $\mathbf{L}^t f(\cdot)$ is again of left-$\delta$-type. Moreover, for each nonnegative integer $n$ the function $\mathbf{L}^t f(\cdot) \in \mc S^p_\delta(X)$. Hence by Lemma \ref{Lem:analytic in the interior} the integral
on the right hand side of (\ref{ali:101}) exists absolutely. We apply Fubini's theorem to interchange the integrals and then we put $x^{-1}k=y^{-1}$ to get:
\begin{align}
\label{ali:102}
(\ref{ali:101})&= (-1)^n \texttt{P} \left( \frac{\partial}{\partial \lambda}\right) \int_K \int_G tr \mathbf{L}^t f(ky)~~ e^{(i \lambda - \rho)\mc H(y^{-1})} \delta(k^{-1}) dy dk \nonumber\\
&= (-1)^n \texttt{P}\left( \frac{\partial}{\partial \lambda}\right) \int_G \left\{ \int_K tr \mathbf{L}^t f(ky) ~~ \delta(k^{-1})dk  \right\} ~~e^{(i \lambda - \rho)\mc H(y^{-1})} dy \nonumber\\
&= \frac{(-1)^t}{d_\delta}  \texttt{P} \left( \frac{\partial}{\partial \lambda}\right) \int_G \mathbf{L}^t f(y)~~ e^{(i \lambda - \rho)\mc H(y^{-1})} dy\nonumber\\
&= \frac{(-1)^t}{d_\delta}  \texttt{P} \left( \frac{\partial}{\partial \lambda}\right) \int_G \mathbf{L}^t f(y^{-1})~~ e^{(i \lambda - \rho)\mc H(y)} dy
\end{align}
The third lie follows by using the Schwartz space extension of the isomorphism $\mc Q$ of Proposition \ref{left delta left delta check identification} and the last line uses the invariance of the Haar measure under the transformation $g \mapsto g^{-1}$.
Let us now break up the group $G$ as well as the Haar measure using the Iwasawa decomposition $KAN$ decomposition to get.
\begin{align}
\label{ali:103}
(\ref{ali:102})=&\frac{(-1)^t}{d_\delta}  \texttt{P} \left( \frac{\partial}{\partial \lambda}\right)  \int_K \int_{\mf a^+} \int_N \mathbf{L}^t f(n^{-1} (\exp H)^{-1} k^{-1}) e^{(i \lambda - \rho)(\mc H(k(\exp H)n)}   dk e^{2\rho(H)} dH dn \nonumber \\
&= \frac{(-1)^t}{d_\delta}  \texttt{P} \left( \frac{\partial}{\partial \lambda}\right)  \int_{\mf a^+} \int_N  \mathbf{L}^t f(n^{-1}(\exp H)^{-1}) e^{(i \lambda + \rho)(H)}    dH dn.
\end{align}
Let $ \alpha_1, \alpha_2, \cdots, \alpha_r$ be the set of all positive restricted roots. Let $\varepsilon_i \in \mathfrak a^*$ ($1 \leq i \leq r$) be such that $\langle\alpha_i, \varepsilon_j\rangle = \delta_{ij}$. Then clearly $\{\varepsilon_i \}_{1 \leq i \leq r}$ forms a basis of $\mathfrak a^*$ and thus we introduce a global coordinate on $\mathfrak a^*$ by $\lambda = \sum_{i=1}^r \lambda_i \varepsilon_i$ ($\forall \lambda \in \mathfrak a^*$).\\
Let $\texttt{P}(\lambda)= \sum_{\theta=0}^\beta \sum_{\beta_1+\cdots+ \beta_r =\theta} \alpha_{\beta_1+\cdots+ \beta_r} \lambda_1^{\beta_1} \lambda_2^{\beta_2} \cdots \lambda_r^{\beta_r}$. Thus,
$$\texttt{P}\left(\frac{\partial}{\partial \lambda}\right) = \sum_{\theta=0}^\beta \sum_{\beta_1+\cdots+ \beta_r =\theta} \alpha_{\beta_1+\cdots+ \beta_r} \left(\frac{\partial}{\partial \lambda_1}\right)^{\beta_1} \left(\frac{\partial}{\partial \lambda_2}\right)^{\beta_2} \cdots \left(\frac{\partial}{\partial \lambda_r}\right)^{\beta_r}.$$
There fore it is easy to check that:
\begin{equation}
\label{eq:3}
\texttt{P}\left(\frac{\partial}{\partial \lambda} \right) e^{(i \lambda + \rho)(H)}= \texttt{P}(H)~ e^{(i \lambda + \rho)(H)},
\end{equation}
where $\texttt{P}(H) =  \sum_{\theta=0}^\beta \sum_{\beta_1+\cdots+ \beta_r =\theta} \alpha_{\beta_1+\cdots+ \beta_r} \varepsilon_1^{\beta_1}(i H)) \cdots \varepsilon_r^{\beta_r}(i H)$.
From (\ref{ali:10}) it follows that:
\begin{align}
\label{ali:11}
&\left\| \texttt{P} \left( \frac{\partial}{\partial \lambda}\right)
\left\{ \left( \langle \lambda, \lambda \rangle_1  + \|\rho\|^2 \right)^n \widetilde{f}(\lambda)\right\} \right\|_{\mathbf{2}} \nonumber\\
&\hspace{1in}\leq \frac{1}{d_\delta} \int_{\mf a^+} \int_N  \|\mathbf{L}^t f(n^{-1}(\exp H)^{-1} )\|_{\mathbf{2}}~ \|\texttt{P}(H)\| \left| e^{(i \lambda + \rho)(H)}\right| dH dn \nonumber\\
&\hspace{1in}\leq \frac{1}{d_\delta} \int_{\mf a^+} \int_A  \|\mathbf{L}^t f(n^{-1}(\exp H)^{-1} )\|_{\mathbf{2}}~ \|\texttt{P}(H)\| ~e^{(|\Im \lambda| +\rho)(H)} dH dn.
\end{align}
 Some basic estimates gives the following:
 \begin{align}
 \label{ali:12}
 \|\varepsilon_j^{\beta_j}(i H)\| &\leq \|\varepsilon_j\|^{\beta_j} \|H\| \nonumber\\
 & \leq c \|\varepsilon_j\|^{\beta_j} |(\exp H) n| \nonumber\\
 & \leq c \|\varepsilon_i\|^{\beta_i} (1+ |(\exp H) n|).
 \end{align}
The above estimate is a consequence of (\ref{eq:Iwasawa < Cartan}) and it is true for all $n \in N$.
 Using (\ref{ali:12}) one can find $d_{\texttt{P}} \in \Z^+$ such that
\begin{equation}
\label{eq:estimate for the polynomial}
\|\texttt{P}(H)\| \leq c_1 (1+ |(\exp H) n|)^{d_{\texttt{P}}}.
\end{equation}
As $f \in \mc S^p_\delta(X)$ so for each $m \in \Z^+$ we have:
\begin{equation}
\label{eq:imp}
\|\mathbf{L}^t f(n^{-1} (\exp H)^{-1})\|_{\mathbf{2}} \leq \mu_{\mathbf{L}^t, m} (f)~(1+ |(\exp H) n|)^{-m}~ \varphi_0^{\frac{2}{p}}(n^{-1} (\exp H)^{-1}).
\end{equation}
The above inequality also uses the fact that $|g|= |g^{-1}|$ for all $g \in G$.
The estimates (\ref{ali:12}) and (\ref{eq:imp}) reduce the inequation (\ref{ali:11}) to the following:
\begin{align}
\label{ali:14}
&\left\| \texttt{P} \left( \frac{\partial}{\partial \lambda}\right)
\left\{ \left( \langle \lambda, \lambda \rangle_1  + \|\rho\|^2 \right)^t \widetilde{f}(\lambda)\right\} \right\|_{\mathbf{2}} \nonumber\\
 &\leq c_1 \frac{\mu_{\mathbf{L}^t, m}(f)}{d_\delta} \int_{\mf a^+} \int_N  \varphi_0^{\frac{2}{p}}(n^{-1}(\exp H)^{-1})~(1+ |(\exp H) n|)^{-m+d_{\texttt{P}}} ~e^{(|\Im \lambda| +\rho)(H)} dH dn \nonumber\\
&= c_1 \frac{\mu_{\mathbf{L}^t, m}(f)}{d_\delta} \int_K \int_{\mf a^+} \int_N  \varphi_0^{\frac{2}{p}}(n^{-1}(\exp H)^{-1} k^{-1})(1+ |k(\exp H)n|)^{-m+d_{\texttt{P}}} \nonumber\\
&\hspace{3.1in} e^{(|\Im \lambda| -\rho)(\mc H(k (\exp H)n))}dk e^{2 \rho(H)}dH dn  \nonumber\\
&= c_\delta \mu_{\mathbf{L}^t, m}(f)\int_G \varphi_0^{\frac{2}{p}}(g^{-1}) (1 + |g|)^{-m + d_{\texttt{P}}} e^{(|\Im \lambda| - \rho)(\mc H(g))} dg
\end{align}
where  $c_\delta= c_1 \frac{1}{d_\delta}$.
Now we use the Cartan decomposition i.e $g = k_1 \exp{|g|} k_2$  and appropriate form of the Haar measure (\ref{eq:Haar measure KAK}) to get:
\begin{align}
\label{ali:15}
(\ref{ali:14})& = c_\delta \mu_{\mathbf{L}^t, m}(f)\int_{\mathfrak a^+} \int_K  \varphi_0^{\frac{2}{p}} (\exp{|g^{-1}|}) (1 + |g|)^{-m+d_{\texttt{P}}} e^{(|\Im \lambda| - \rho)(\mc H( \exp{|g|} k_2))}\Delta(|g|) d |g|~ d k_2, \nonumber\\
& =c_\delta \mu_{\mathbf{L}^t, m}(f) \int_{\mathfrak a^+}\hspace{-.1in}\varphi_0^{\frac{2}{p}} (\exp{|g|}) (1 + |g|)^{-m+d_{\texttt{P}}}\left\{ \int_K e^{(i(-i|\Im \lambda|) - \rho)(\mc H( \exp{|g|} k_2))} d k_2\right\} \nonumber\\
&\hspace{4.5in} \Delta(|g|) d|g|\nonumber\\
&=c_\delta \mu_{\mathbf{L}^t, m}(f) \int_{\mathfrak a^+ } \varphi_0^{\frac{2}{p}} (\exp{|g|}) (1 + |g|)^{-m+d_{\texttt{P}}} ~ \varphi_{-i |\Im \lambda |} (\exp|g^{-1}|) \Delta(|g|) d|g| \nonumber\\
& \leq c_\delta \mu_{\mathbf{L}^t, m}(f) \int_{\mathfrak a^+ } \varphi_0^{\frac{2}{p}+1} (\exp{|g|}) (1 + |g|)^{-m+d_{\texttt{P}}} ~e^{|\Im \lambda(|g|)|}  \Delta(|g|) d|g|
\end{align}
where the last inequality in this chain follows by using the estimate (\ref{eq:estimate of phi lambda}) of the elementary spherical function. We take $\lambda \in \mf a^*_\e$, therefore $|\Im \lambda(|g|)| \leq \e \rho(|g|) $ where $\e = \left(\f{2}{p}-1 \right)$. Now by using the another fundamental estimate
(\ref{eq:estimate of phi_0}) we further  reduce the inequality (\ref{ali:15})  to the following
\begin{align}
\label{ali:16}
&\left\| \texttt{P} \left( \frac{\partial}{\partial \lambda}\right)
\left\{ \left( \langle \lambda, \lambda \rangle_1  + \|\rho\|^2 \right)^t \widetilde{f}(\lambda)\right\} \right\|_{\mathbf{2}} \nonumber\\
 &\hspace{1.5in}\leq c_\delta \mu_{\mathbf{L}^t, m}(f) \int_{\mathfrak a^+} \varphi_0^2(\exp{|g|}) (1 + |g|)^{-m +d_{\texttt{P}} + \e \theta} \Delta(|g|) ~d|g|, \nonumber\\
&\hspace{1.5in} = c_\delta \mu_{\mathbf{L}^t, m}(f) \int_{G} \varphi_0^2(g) ~(1 + |g|)^{-m +d_{\texttt{P}} + \e \theta} ~dg.
\end{align}
We choose a suitably large  $m \in \Z^+$  so the integral (\ref{ali:16}) converges (\cite[Lemma 7]{Harish66}). This completes the proof of the Lemma.
\end{proof}
We now take up the extension of the inversion formula (\ref{eq:inversion formuta delta sp. transform}) of the $\delta$-spherical transform.
\begin{Lemma}
\label{Lem:inversion exists in SW space lavel}
  For each $h \in \mc S_\delta(\mathfrak a^*_\e)$ the following integral
\begin{equation}
\label{eq:inv formula sw}
\frac{1}{\omega} \int_{\mathfrak a^*} \Phi_{\lambda, \delta}(x) ~h(\lambda)~|\mathbf{c}(\lambda)|^{-2} d \lambda, ~~\mbox{~~~~for } x\in X,
\end{equation}
gives a $Hom(V_\delta, V_\delta)$ valued, left-$\delta$-type $\mc C^\infty$ function on $X$. (From now on we shall denote this function by $\mathcal I h(\cdot)$.)
\end{Lemma}
\begin{proof}
  Let us take any $D \in \mathcal U(\mathfrak g_\C)$. Then,
\begin{align}
\label{ali:7}
\frac{1}{\omega} \int_{\mathfrak a^*} \left\|\Phi_{\lambda, \delta}(D, x) \right\|_{\mathbf{2}}& \|h(\lambda)\|_{\mathbf{2}} |\mathbf{c}(\lambda)|^{-2} d \lambda \leq c_\delta ~\varphi_0(x) \int_{\mathfrak a^*} (1+ \|\lambda\|)^{b_{\textbf{D}}+b-n} d \lambda.
\end{align}
The above inequality follows from  the fact that $h \in \mc S_\delta(\mathfrak a^*_\e)$  and by using the decay (\ref{eq:seminorm on S delta imageside}), the estimate (\ref{eq:c-function estimate}) and the estimate (\ref{eq:pri:estimate of gen sp funct}) for the generalized spherical functions.
One can choose a suitably large $n$ so that the integral in the right hand side of (\ref{ali:7}) converges. This proves $\mathcal I h$ is a function on $X$ and $\textbf{D} \mathcal I h$ exists for all $\textbf{D} \in \mathcal U(\mathfrak g_\C)$. Hence $\mathcal I h \in \mc C^\infty(X, Hom(V_\delta, V_\delta))$. As, $\Phi_{\lambda, \delta}(\cdot)$ is of left-$\delta$-type ( (\ref{eq:pri:behaviour of gen. sp. funct. under K action}) of Proposition \ref{prop:pri:basic properties of gen, sp. funct.}), so is $\mathcal I h$.
\end{proof}
\begin{Lemma}
\label{Lem:cont. of the inverse of delta sp trans}
  If $h \in \mc S_\delta(\mathfrak a^*_\e)$ then the inverse $\mathcal I h \in \mc S^p_\delta(X)$.
\end{Lemma}
\begin{proof}
To prove this Lemma we shall first consider the spaces $ \mc P^\delta(\mf a^*_\C)$ and $\mc D_\delta(X)$ equipped with the topologies of the respective Schwartz spaces containing them. We have already noticed that $\mc P^\delta(\mf a^*_\C)$ and $\mc D_\delta(X)$ are dense subspaces of $\mc S_\delta(\mf a^*_\e)$ and $\mc S^p_\delta(X)$ respectively.\\
We shall show that  $\mathcal I $ is a continuous map from $\mathcal P^\delta(\mathfrak a^*_\C)$ onto (by Theorem \ref{The:Topological PW Theorem}) $\mc D_\delta(X)$ with respective to the Schwartz space topologies. That is for  $h \in \mathcal P^\delta(\mathfrak a^*_\C)$ and for each seminorm $\mu$ on $\mc D_\delta(X)$ , there exists a seminorm $\nu$ on $\mc P^\delta(\mf a^*_\C)$ such that $\mu(f) \leq c_\delta \nu(h)$, where $f = \mathcal I h \in \mc D_\delta(X)$ and $c_\delta$ is a positive constant depending on $\delta \in \widehat{K}_M$.\\
As, $f \in \mc D_\delta(X)$, by Lemma \ref{Lem:5}, we get a  function $\phi \in \mathcal D( G //K, Hom(V_\delta, V_\delta))$ such that $f \equiv \textbf{D}^\delta \phi$. If $\Phi$ be the image of $\phi$ under the spherical transform then it follows easily that $h= Q^\delta \Phi$. Let $\textbf{D}, \textbf{E} \in \mathcal U(\mathfrak g_\C)$ and $n$ be any nonnegative integer, then
\begin{align}
\label{ali:8}
\mu_{\textbf{D}, \textbf{E}, n}(f) &= \sup_{x \in G} \|f(\textbf{D}, x, \textbf{E})\|_{\mathbf{2}} (1 + |x|)^n \varphi^{- \frac{2}{p}}_0(x), \nonumber\\
&=  \sup_{x \in G} \|\textbf{D}^\delta \phi (\textbf{D}, x, \textbf{E})\|_{\mathbf{2}} (1 + |x|)^n \varphi^{- \frac{2}{p}}_0(x), \nonumber\\
&= {\mu_0}_{\textbf{D}^\delta \textbf{D}, \textbf{E}, n }(\phi).
\end{align}
(~~\emph{Here $\mu_0$ denote the seminorms on the Fr\'{e}chet space $\mc S^p(G // K, Hom(V_\delta, V_\delta))$}.~~) At this point we use Anker's \cite{Anker91} proof of the Schwartz space isomorphism theorem for bi-$K$-invariant functions. For each $\textbf{D}, \textbf{E} \in \mathcal U(\mathfrak g_\C)$ and $n \in \Z^+$ one can find a polynomial $P \in S(\mf a)$ and $m_\delta \in \Z^+$ (depending on $d_\delta$) such that,
\begin{align}
\label{ali:9}
{\mu_0}_{\textbf{D}^\delta \textbf{D}, \textbf{E}, n } (\phi) &\leq c_\delta \sup_{\lambda \in Int \mathfrak a^*_\e} \left\| P\left( \frac{\partial}{\partial \lambda}\right) \Phi(\lambda) \right\|_{\mathbf{2}} (1 + \|\lambda\|)^{m_\delta}, \nonumber\\
& \leq c_\delta \sup_{\lambda \in Int \mathfrak a^*_\e} \left\| P_1\left( \frac{\partial}{\partial \lambda}\right) h(\lambda) \right\|_{\mathbf{2}} (1 + \|\lambda\|)^{m^\prime_\delta}.
\end{align}
The last line in (\ref{ali:9}) follows by using the isomorphism, proved in Lemma \ref{Lem:S-delta and S-0 homeomorphic},  between the Schwartz spaces $S_0(\mathfrak a^*_\e)$ and $S_\delta(\mathfrak a^*_\e)$. Hence (\ref{ali:8}) and (\ref{ali:9}) togather gives $\mu_{\textbf{D}, \textbf{E}, n}(f) \leq c_\delta \nu_{P_1, m^\prime_\delta}(h)$.
As we have started with an $h \in \mathcal P^\delta(\mathfrak a^*_\C) \subset \mc S_\delta(\mathfrak a^*_\e)$, the right hand side of (\ref{ali:9}) is clearly finite. Hence $\mathcal I h= f \in \mc S^p_\delta(X)$.\\
Now we apply the density argument to conclude the Lemma. Let us now take $h \in \mc S_\delta(\mathfrak a^*_\e) $. As, $\mathcal P_\delta(\mathfrak a^*_\C)$ is dense in $\mc S_\delta(\mathfrak a^*_\e)$, there exists a Cauchy sequence $\left\{h_n \right\} \subset \mathcal P_\delta(\mathfrak a^*_\C)$ converging to $h$. Then, by what we have proved above, we can get a Cauchy sequence $\left\{ f_n\right\} \subset \mc D_\delta(X)$ such that $\widetilde{f_n} = h_n$. As $\mc S^p_\delta(X)$ is a Fr\'{e}chet space the sequence must converge to some $f \in \mc S^p_\delta(X)$. Clearly, $f = \mathcal I h $. This completes the proof of the Lemma.
\end{proof}
We note that, the Lemma \ref{Lem:cont. of the inverse of delta sp trans} also implies  the fact that the $\delta$-spherical transform  is an injection in the corresponding Schwartz space level. \\
Finally, Lemma \ref{Lem:continuity of the delta spherical transform} and Lemma \ref{Lem:cont. of the inverse of delta sp trans} together shows that the $\delta$-spherical transform is a continuous surjection of $\mc S_\delta^p(X)$ onto $\mc S_\delta(\mf a^*_\e)$ for $(0 < p \leq 2)$. A simple application of the open mapping theorem concludes that the $\delta$-spherical transform is a topological isomorphism between the corresponding Schwartz spaces. This proves the Theorem \ref{Theorem:main-1}.\\  In the next section we shall extend this result to a slightly larger class of functions.

\section{\textbf{ Finite $K$-type functions}}\label{sec:Finite-type-functio}
\setcounter{equation}{0}
Let choose and fix a finite subset $\Gamma \subset \what{K}_M$. We denote $\mc D_\Gamma(X)$ for the space of all compactly supported $\mc C^\infty$ functions on $X$ with the property that: for $f \in\mc D_\Gamma(X)$~$f^\delta \equiv 0$ for $\delta \notin \Gamma $.
We take the subclass $\mc S_\Gamma^p(X)$ of the Schwartz class $\mc S^p(X)$ (for $0< p \leq 2$) defined by
\begin{equation}
 \label{eq: def left finite schwartz space}
 \mc S^p_\Gamma(X) = \{ f \in \mc S^p(X ) ~| ~ f(X) = \sum_{\delta \in \Gamma} tr f^\delta (x)~ \mbox{~for  all~} x \in X \}.
 \end{equation}
 The seminorms on $\mc S^p_\Gamma(X)$ are  as follows: for each $\textbf{D}, \textbf{E} \in \mc U(\mf g_\C)$ and $n \in \Z^+$,
 \begin{equation}
 \label{eq:seminorm-K-finite}
 {\mu_\Gamma}_{\textbf{D}, \textbf{E}, n} (f) = \sup_{\delta \in \Gamma, x \in X } \left\| f^\delta(\textbf{D}, x, \textbf{E}) \right\|_{\mathbf{2}}~(1 + |x|)^n~ \varphi_0^{- \frac{2}{p}}(x)~< +\infty.
\end{equation}
Clearly $\mc D_\Gamma(X)$ is a dense subset of the Schwartz space $\mc S^p_\Gamma(X)$ with respect to the Fr\'{e}chet topology induced by the countable family of seminorms $\{{\mu_\Gamma}_{\textbf{D}, \textbf{E}, n}\}$.
  It also follows easily from the definition \ref{def:Classical L-p Schwartz space} and (\ref{eq:op-valued delta Sw-space decay} ) that, if    $f \in  \mc S^p_\Gamma(X)$ then for each  $\delta \in \Gamma $ the projection $f^\delta \in \mc S^p_\delta(X)$.  For these classes of functions  the transform we shall mainly consider is the Helgason Fourier transform.\\
Let us now define the Schwartz class functions on the domain $\mf a^*_\e \times K/M$.
\begin{Definition}
 \label{def: HFT image Schwartz space}
   Let $\mc S_\Gamma(\mf a^*_\e \times K/M)$ denotes the class of functions $h$ on $\mf a^*_\e \times K/M$ satisfying the following properties:
   \begin{enumerate}
     \item For each $kM \in K/M$, the function $\lambda \mapsto h(\lambda, kM )$ is holomorphic on $Int \mf a^*_\e$, and it extends as a continuous function on the closed complex tube $\mf a^*_\e$. The function $h$ is a smooth function in the $k \in K/M$ variable.
      \item For all $\lambda\in \mathfrak a^*_\e$, $\omega \in W$ and $x \in G$
   \begin{equation}
   \check{h}(\lambda, x) = \check{h} (\omega \lambda, x),
   \end{equation}
   where $\check{h}(\lambda, x) = \int_K h(\lambda, k) e^{-(i \lambda+ \rho)H(x^{-1}k)} dk $.
     \item For each $P \in S (\mf a)$  and for  integers $n, m>0$ the function $h$ satisfies the following decay condition
  \begin{equation}
  \label{eq: SW decay for image HFT}
  \sup_{(\lambda, k)\in Int \mf a^*_\e \times K/M} \left| P\lt(\f{d}{d \lambda}\rt) h(\lambda, k, \omega^m_{\mf k}) \right| ~(1 + |\lambda|)^n < ~+ \infty.
  \end{equation}
  \item For each $\delta \in \what{K}_M \setminus \Gamma$ the left-$\delta$-projection  $h^\delta$ defined by
        \begin{equation}
        \label{eq:ch2:delta projection of h}
        h^\delta(\lambda, k) = d_\delta\int_K h(\lambda, k_1 k ) \delta(k_1^{-1}) dk_1,
        \end{equation}
        is identically a zero function on $\mf a^*_\e \times K/M$.
   \end{enumerate}
  \end{Definition}
  The space $\mc S_\Gamma(\mf a^*_\e \times K/M)$ is a Fr\'{e}chet space with the topology induced by the seminorms (\ref{eq: SW decay for image HFT}). By the theory of smooth functions on compact groups \cite{Sugiura71}, the topology of the space $\mc S_\Gamma(\mf a^*_\e \times K/M)$ can be given by the following equivalent family of seminorms, for each $P \in S(\mf a)$ and $m \in \Z^+$ we have
     \begin{equation}
     \sup_{\lambda \in Int \mf a^*_\e, \delta \in F} \lt\|P\lt(\hspace{-.03in}\f{d}{d \lambda} \hspace{-.03in}\rt) h^\delta(\lambda, eM) \rt\|_{\textbf{2}} (1+ |\lambda|)^m < + \infty, \mbox{~for~} h \in \mc S_\Gamma(\mf a^*_\e \times K/M).
     \end{equation}
     We denote by $\mc S(\mf a^*_\e \times K/M)$ the Fr\'{e}chet space satisfying all the conditions of the Definition \ref{def: HFT image Schwartz space} except condition (iv). The space $\mc S_\Gamma(\mf a^*_\e \times K/M)$  is a closed subspace of $\mc S(\mf a^*_\e \times K/M)$. We know that the HFT can extended to the  Schwartz class $\mc S^p(X)$  \cite{Eguchi76}, furthermore the HFT is a continuous map from $\mc S^p(X)$ into $\mc S(\mf a^*_\e \times K/M)$. Hence the HFT is a continuous map from $\mc S^p(F, X)$ into $\mc S_\Gamma(\mf a^*_\e \times K/M)$.
  \begin{Lemma}
  \label{lem:ch2: hdelta}
  Let $h \in \mc S_\Gamma(\mf a^*_\e \times K/M)$, then for each $\delta \in F$, the left-$\delta$-projection $h^\delta \in \mc S_\delta(\mf a^*_\e)$.
  \end{Lemma}
  \begin{proof}
  The function $\lambda \mapsto h^\delta(\lambda, eM)$ trivially satisfies condition (i) of Definition \ref{def: delta SW space in the image side}. The required decay (\ref{eq:Sw. space decay on lambda}) is also an easy consequence of (\ref{eq: SW decay for image HFT}). It can be shown that the $\delta$-projection $h^\delta$ also satisfies the condition $$\check{(h^\delta)}(\lambda, x) = \check{(h^\delta)}( \omega \lambda, x) \mbox{~for all~} \omega \in W.  $$
  It is easy to check that $h^\delta(\lambda, kM) = \delta(k) h^\delta(\lambda, eM)$, hence for each $(\lambda, a) \in \mf a^*_\e \times A$ we write $\check{(h^\delta)}(\lambda, a)$  as follows
  $$\check{(h^\delta)}(\lambda, a) = \Phi_{\lambda, \delta}(a) h^\delta(\lambda, eM).$$
  By the property (\ref{W action on Phi|A}) of the generalized spherical functions it follows that the function $\lambda \mapsto Q^\delta(\lambda)^{-1}h^\delta(\lambda, eM)$ is $W$-invariant. Hence we conclude that $h^\delta(\cdot, eM) \in \mc S_\delta(\mf a^*_\e)$.
  \end{proof}
By using Theorem \ref{Theorem:main-1}, for each $h \in \mc S_\Gamma(\mf a^*_\e \times K/M)$ we get an unique finite sequence $\{ f^\delta\}_{\delta \in \Gamma}$ of $\mc C^\infty $ functions on $X$ such that each member $f^\delta \in \mc S^p_\delta(X)$. We consider the following scalar valued function
\begin{equation}
\label{eq:ch2:function}
f(x) = \sum_{\delta \in F} trf^\delta(x), ~~~x \in X.
\end{equation}
 For each $\delta \in \Gamma$, $\lt(\mc F^{-1}h \rt)^\delta(x) = \mc I(h^\delta)(x) = f^\delta(x)$.  Hence, we get $\mc F^{-1} h(x) = f(x) $ for all $x \in X$. The function $f \in \mc S^p_\Gamma(X)$. Furthermore for each $D, E\in  \mc U(\mf g_\C)$ and $n \in \Z^+$ we have
\begin{align*}
\sup_{x \in G} |f(D,x,E)| &(1+|x|)^n \varphi_0^{-\f{2}{p}}(x) \\ &\leq c \sup_{x \in G, \delta \in \Gamma} \|f^\delta(D,x,E)\|_{\textbf{2}} (1+|x|)^n \varphi_0^{-\f{2}{p}}(x)\\
&\leq c_1  \sup_{\lambda \in Int \mf a^*_\e, \delta \in \Gamma} \lt\|P\lt(\f{d}{d \lambda}\rt) h^\delta(\lambda, eM) \rt\|_{\textbf{2}} (1 + |\lambda|)^m\\
&\leq c_2 \sup_{\lambda \in Int \mf a^*_\e, k \in K} \lt|P_1\lt(\f{d}{d \lambda}\rt) h(\lambda, k, \omega_{\mf k}^r)\rt| (1 + |\lambda|)^{m_1},
\end{align*}
for some $P_1 \in \mc S (\mf a^*)$ and $r, m_1 \in \Z^+$.
 Thus, the HFT is a bijective map from $\mc S^p_\Gamma(X)$  to $\mc S_\Gamma(\mf a^*_\e \times K/M)$. Once again, by the open mapping theorem we conclude the following.
\begin{Theorem}
\label{Theo:main theorem 2}
Let $\Gamma$ be a finite subset of $\what{K}_M$, then the HFT is a topological isomorphism of the space $\mc S^p_\Gamma(X)$ onto the Fr\'{e}chet space $\mc S_\Gamma(\mf a^*_\e \times K/M)$.
\end{Theorem}
\vspace{.2in}
\textbf{Acknowledgments:}~ The author is thankful to
 Prof. Angela Pasquale of Universit\'{e} de Metz, France for her valuable suggestions.

\bibliographystyle{amsalpha}
\bibliography{mybib}
\end{document}